\documentclass[11pt]{article}
\usepackage[utf8]{inputenc}
\usepackage{amsthm}
\usepackage{amsfonts}
\usepackage{amsmath}
\usepackage{amssymb}
\usepackage{colonequals}
\usepackage{epsfig}
\usepackage{url}
\usepackage{asymptote}
\usepackage{hhline}
\usepackage[colorlinks]{hyperref}

\newcommand{\CC}{{\cal C}}

\newtheorem{theorem}{Theorem}
\newtheorem{conjecture}[theorem]{Conjecture}
\newtheorem{corollary}[theorem]{Corollary}
\newtheorem{lemma}[theorem]{Lemma}

\newtheorem{observation}[theorem]{Observation}

\begin{asydef}

void vertex(pair a, pen barva=white, real pol = 0.08)
{
  filldraw(circle (a, pol), fillpen=barva);
}

void cvertex(pair a, pen barva, string s)
{
  filldraw(circle (a, 0.25), fillpen=white, drawpen=barva);
  label(origin=a, L=scale(0.7) * Label(s), p=barva);
}

pair v[];
int i,j;
\end{asydef}

\title{On a conjecture concerning 4-coloring of graphs with one crossing}
\author{Zdeněk Dvořák\thanks{Computer Science Institute, Charles University, Prague, Czech Republic. E-mail: \url{rakdver@iuuk.mff.cuni.cz}.
Supported by the ERC-CZ project LL2328 (Beyond the Four Color Theorem, registered as ``Realizace projektu hraničního výzkumu v oblasti teorie grafů -- barevnost a návrh algoritmů'') of the Ministry of Education of Czech Republic.}
\and Bernard Lidický\thanks{Department of Mathematics, Iowa State University. Ames, IA, USA. E-mail: \url{lidicky@iasate.edu}.
Supported in part by NSF DMS-2152490.
This material is based upon work supported by the National Science
Foundation under Grant No. DMS-1928930, while the author was in
residence at the Simons Laufer Mathematical Sciences Institute in
Berkeley, California, during the Spring 2025 semester.
}
\and Bojan Mohar\thanks{Department of Mathematics, Simon Fraser University, Burnaby, BC V5A 1S6, Canada. E-mail: \url{mohar@sfu.ca}.
Supported in part by the NSERC Discovery Grant R832714 (Canada),
by the ERC Synergy grant (European Union, ERC, KARST, project number
101071836),
and by the Research Project N1-0218 of ARIS (Slovenia). On leave
from
FMF, Department of Mathematics, University of Ljubljana, Ljubljana, Slovenia.}}
\date{}

\begin{document}
\maketitle

\begin{abstract}
We conjecture that every graph of minimum degree five with no separating triangles
and drawn in the plane with one crossing is 4-colorable.  In this paper, we use computer enumeration
to show that this conjecture holds for all graphs with at most 28 vertices, explore the consequences
of this conjecture and provide some insights on how it could be proved.
\end{abstract}

Famously, every planar graph is 4-colorable~\cite{AppHak1,AppHakKoc}.  It is natural to ask whether
this statement can be strengthened.  As shown in~\cite{Fisk78,hutcheuler},
for every surface $\Sigma$ other than the sphere, there are infinitely many minimal non-4-colorable, i.e.,
\emph{$5$-critical}, graphs which can be drawn in $\Sigma$ (a graph $G$ is \emph{$k$-critical} if $\chi(G)=k$
and every proper subgraph of $G$ is $(k-1)$-colorable).  Still, one could hope that at
least for simplest non-trivial surfaces $\Sigma$ (projective plane, torus), it might be possible to
characterize 5-critical graphs drawn in $\Sigma$, or at least to develop a polynomial-time algorithm
to test 4-colorability of graphs drawn in $\Sigma$.  However, either of these goals seems
far beyond our reach using the current methods.

To obtain at least some intuition, we consider ``minimally nonplanar'' graphs, or more precisely, graphs that
can be drawn in the plane with at most one crossing; let $\CC$ denote the class of all such graphs.  Let us remark that graphs
in $\CC$ can be drawn without crossings both in the projective plane and on the torus.  It is still easy to construct
infinite families of 5-critical graphs in $\CC$; however, preliminary computer-assisted investigations suggest
that obtaining their characterization could be possible.  In this paper, we explore in detail a natural conjecture
arising from these investigations.

Of course, 5-critical graphs have minimum degree at least four.  Moreover, there is a natural way how to handle
vertices of degree four: Suppose $G\neq K_5$ is a 5-critical graph and a vertex $v\in V(G)$ has degree four.
Since $G$ is 5-critical and $G\neq K_5$, the graph $G$ does not contain $K_5$ as a subgraph, and thus
$v$ has distinct neighbors $x$ and $y$ such that $xy\not\in E(G)$.  Let $G'$ be the graph obtained from $G-v$ by identifying $x$ with $y$ to a single vertex.
Assuming that $G$ can be drawn in the plane with one crossing so that this crossing is not on the edges $vx$ and $vy$,
the graph $G'$ can also be drawn in the plane with at most one crossing.  Moreover, it is easy to see that any $4$-coloring of $G'$
could be extended to a 4-coloring of $G$. Since $G$ is $5$-critical, it follows that $G'$ is not $4$-colorable, and thus it has a 5-critical
subgraph $G''\in \CC$.  Based on this observation, one could hope to inductively prove statements about the structure of 5-critical graphs in $\CC$
of minimum degree four.  Indeed, a similar inductive argument used to handle 4-faces has been crucial in the development of the theory of 4-critical triangle-free
graphs, which eventually lead to design of a linear-time algorithm for 3-colorability of triangle-free graphs on surfaces~\cite{trfree7}.

For this argument to take off, one needs to handle the basic case of 5-critical graphs in $\CC$ of minimum degree at least five.
Based on computer-assisted experiments, we believe no such graphs exists, and thus $K_5$ is the sole basic case for the argument described above.
\begin{conjecture}\label{conj-main}
Every 5-critical graph in $\CC$ contains a vertex of degree four.
\end{conjecture}
Indeed, the following stronger statement seems to hold (a subgraph $K$ of a graph $G$ is \emph{separating} if $G-V(K)$ is not connected).
\begin{conjecture}\label{conj-main-strong}
If a graph $G$ does not contain separating triangles and has a drawing in the plane with one crossing such that all vertices not incident with the crossed edges
have degree at least five, then $G$ is 4-colorable.
\end{conjecture}
Let us remark that throughout the paper, the graphs are simple, without loops and parallel edges.
It is easy to see that $5$-critical graphs do not contain separating cliques,
and thus Conjecture~\ref{conj-main-strong} implies Conjecture~\ref{conj-main}.
The assumption that $G$ does not contain separating triangles cannot be dropped: Every graph $G\in \CC$ has a supergraph of minimum degree
at least five in $\CC$, obtained by gluing a copy of the icosahedron at each vertex of $G$.  A similar argument shows that Conjecture~\ref{conj-main-strong} has
the following consequence.
\begin{lemma}\label{lemma-wheel}
Let $G\in \CC$ be a graph without separating triangles, and consider any drawing of $G$ in the plane with at most one crossing.
If $G$ is not 4-colorable and Conjecture~\ref{conj-main-strong} holds, then there exists a vertex $v\in V(G)$ of degree four such that all incident faces have length three
and their boundaries do not contain the crossing.
\end{lemma}
In particular, this has the following consequence (where $W_k$ is the \emph{$k$-wheel} graph consisting of a $k$-cycle and a vertex adjacent to all vertices of this cycle).
\begin{corollary}\label{cor-wheel}
If Conjecture~\ref{conj-main-strong} holds, then every graph in $\CC$ that does not contain $W_4$ as a subgraph is $4$-colorable.
\end{corollary}
It is well known that it suffices to prove the Four Color Theorem for triangulations of the plane.  Similarly, Conjecture~\ref{conj-main-strong} has the following
equivalent form.  Let $\CC_0\subseteq \CC$ consist of graphs obtained from 4-connected plane triangulations by choosing distinct facial triangles $xuv$ and $yuv$ sharing
an edge $uv$ and adding the edge $xy$; we say that the vertices $x$, $y$, $u$, and $v$ are \emph{external} and all other vertices \emph{internal}.
\begin{conjecture}\label{conj-main-strong-tria}
Every non-4-colorable graph in $\CC_0$ has an internal vertex of degree four.
\end{conjecture}
Assuming Conjecture~\ref{conj-main-strong} holds, we can show that every non-4-colorable graph in $\CC_0$
can be obtained from $K_5$ by a sequence of simple operations.  Since we want to avoid creating separating triangles, this is not completely straightforward;
see Section~\ref{sec-gen4} for the definition of the \emph{expansion} operation.
\begin{theorem}\label{thm-gen4}
If Conjecture~\ref{conj-main-strong} holds, then for every non-$4$-colorable graph $G\in \CC_0$, there
exists a sequence $K_5=G_0, G_1,\ldots, G_m=G$ of graphs such that for every $i\in\{1,\ldots,m\}$, the graph $G_i\in \CC_0$
is non-$4$-colorable and is obtained from $G_{i-1}$ by expansion.
\end{theorem}
The expansion operation increases the number of vertices, and thus Theorem~\ref{thm-gen4} can be turned into a relatively
efficient algorithm to enumerate the non-4-colorable graphs in $\CC_0$.
However, it needs to be noted that there are 5-critical graphs in $\CC$ that do not belong to $\CC_0$; see Figure~\ref{fig-excritque}
for an example.  Consequently, Theorem~\ref{thm-gen4} is somewhat less explicit than it might appear, since one of the options for the
expansion operation needs to be essentially ``if $G_{i-1}$ contains the subgraph depicted in Figure~\ref{fig-excritque}, then
$G_i$ can be obtained by replacing the part of $G_{i-1}$ drawn inside the face $h$ by any larger triangulation without separating triangles''.
\begin{figure}
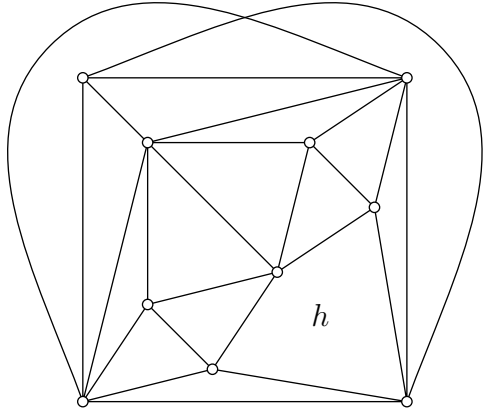

\begin{center}
\asyinclude[width=0.5\textwidth]{fig-excritque.asy}
\end{center}
\caption{A $5$-critical graph in $\CC$ with a $4$-face $h$.}\label{fig-excritque}
\end{figure}

The rest of the paper is organized as follows.  In Section~\ref{sec-toprecol}, we give some basic facts on 5-critical graphs in $\CC$, relating
them to plane graphs which are not 4-colorable so that the coloring of a specific 4-face matches a prescribed pattern,
and prove Lemma~\ref{lemma-wheel} and the equivalence of Conjectures~\ref{conj-main-strong} and \ref{conj-main-strong-tria}.  In Section~\ref{sec-enum-cand},
we give results on enumeration of planar graphs without separating triangles and with all vertices of degree at most four incident
with a specific face.  Using these results, we give a computer-assisted proof that Conjecture~\ref{conj-main-strong} holds for graphs with at most 28 vertices.
In Section~\ref{sec-gen4}, we prove Theorem~\ref{thm-gen4}.  Finally, in Section~\ref{sec-proofcom}, we investigate
how the proof of the Four Color Theorem could be modified to prove Conjecture~\ref{conj-main-strong}.

\section{Reduction to precoloring extension}\label{sec-toprecol}

Let us start with a simple observation.
\begin{lemma}\label{lemma-topre4}
If a graph $G\in\CC$ is not 4-colorable, then $G$ is non-planar and for any drawing of $G$ in the plane with one crossing,
the vertices of $G$ incident with the crossing edges induce a clique.
\end{lemma}
\begin{proof}
The graph $G$ is clearly non-planar by the Four Color Theorem.  Consider any drawing of $G$ in the plane with one crossing,
let $u_1v_1$ and $u_2v_2$ be the crossed edges in this drawing, and suppose for a contradiction that $\{u_1,u_2,v_1,v_2\}$
is not a clique in $G$.  By symmetry, we can assume that the vertices $u_1$ and $u_2$ are non-adjacent in $G$. 
Let $G'$ be the graph obtained from $G$ by identifying $u_1$ with $u_2$ to a single vertex $u$.
Observe that $G'$ is planar, and thus it has a proper 4-coloring.  The 4-coloring
of $G'$ corresponds to a 4-coloring of $G$ obtained by giving $u_1$ and $u_2$ the color of $u$.
This is a contradiction, since $G$ is non-4-colorable.
\end{proof}

Consider a graph $G$ drawn in the plane with one crossings, let $u_1v_1$ and $u_2v_2$ be the crossed edges in this drawing, and
suppose $\{u_1,u_2,v_1,v_2\}$ is a clique in $G$.  We can redraw the edges $u_1u_2$, $u_2v_1$, $v_1v_2$, and $v_2u_1$ next
to the edges $u_1v_1$ and $u_2v_2$ without introducing any crossings, so that the 4-cycle $C=u_1u_2v_1v_2$ bounds a face
of the corresponding plane drawing of the graph $G_0=G-\{u_1v_1,u_2v_2\}$.  Clearly any 4-coloring of $G$ assigns pairwise different colors
to the vertices of $C$, and thus $G$ is $4$-colorable if and only if there exists a 4-coloring of $G_0$ using all four colors on $C$.

The \emph{type} of a coloring $\psi$ of a graph $F$ is the equivalence $\approx$ such that for $u,v\in V(F)$, we
have $u\approx v$ if and only if $\psi(u)=\psi(v)$.
Note that there are exactly four types of $4$-colorings of a 4-cycle $K=x_1x_2x_3x_4$; such a 4-coloring $\varphi$ is
\begin{itemize}
\item \emph{rainbow} if $\varphi(x_1)\neq\varphi(x_3)$ and $\varphi(x_2)\neq\varphi(x_4)$,
\item \emph{$x_i$-diagonal} for $i\in \{1,2\}$ if $\varphi(x_i)=\varphi(x_{i+2})$ and $\varphi(x_{3-i})\neq\varphi(x_{5-i})$, and
\item \emph{bichromatic} if $\varphi(x_1)=\varphi(x_3)$ and $\varphi(x_2)=\varphi(x_4)$.
\end{itemize}
Let $H$ be a plane graph and let $F$ be the subgraph of $H$ consisting of the vertices and edges incident with the outer face
of $H$.  For a type $\approx$ of colorings of $F$, we say that a $4$-coloring $\psi$ of $H$ \emph{matches} $\approx$ if
$\approx$ is the type of the restriction of $\psi$ to $F$.  We say that a graph $H$ is \emph{$\approx$-forbidding} if $H$ does not
have any 4-coloring matching $\approx$.  Equivalently, for any 4-coloring $\varphi$ of $F$, if the type of $\varphi$ is $\approx$,
then $\varphi$ does not extend to a 4-coloring of $H$.

A plane graph $G$ with the outer face bounded by a 4-cycle $K$ is \emph{rainbow-forbidding} if $G$ is $\approx$-forbidding for the type $\approx$ of rainbow colorings of $K$.
Similarly, we define the notion of an \emph{$x$-diagonal-forbidding} graph for $x\in V(K)$ and of a \emph{bichromatic-forbidding} graph.
With these definitions in place, we can restate Lemma~\ref{lemma-topre4} as follows.

\begin{corollary}\label{cor-toext}
A graph $G$ is not 4-colorable and belongs to $\CC$ if and only if there exists a rainbow-forbidding plane graph $H$ with the outer face bounded
by a $4$-cycle $u_1u_2v_1v_2$ such that $G$ is isomorphic to the graph obtained from $H$ by adding the edges $u_1v_1$ and $u_2v_2$.
\end{corollary}

Thus, the problem of characterizing non-4-colorable graphs in $\CC$ is equivalent to the rainbow precoloring extension problem
in plane graphs with the outer face of size four.  Next, let us note a standard observation on triangulating planar graphs.
\begin{lemma}\label{lemma-cantri}
Let $H$ be a plane graph such that every triangle bounds a face and the outer face has length greater than three.
Then there exists a supergraph $H'$ of $H$
such that $V(H')=V(H)$, the outer face of $H$ is also the outer face of $H'$, all internal faces of $H'$ are triangles,
and every triangle in $H'$ bounds a face.
\end{lemma}
\begin{proof}
Let us argue that if $f$ is an internal face of $H$ of length at least four, then there exist
distinct non-adjacent vertices $u$ and $v$ incident with $f$ such that every triangle in $H+uv$ (with the edge $uv$ drawn inside $f$)
bounds a face.
\begin{itemize}
\item If the boundary of $f$ is disconnected, then we can simply add an edge between different components; such an edge clearly
is not contained in any triangle in the resulting graph, and thus every triangle still bounds a face.
\item Otherwise, let $W=v_1v_2\ldots v_k$ be the closed walk bounding
$f$.  If say $v_2$ is a cut vertex separating $v_1$ from $v_3$ in $H$, then we can let $u=v_1$ and $v=v_3$; the only triangle
in $H+v_1v_3$ containing the edge $v_1v_3$ is $v_1v_2v_3$, and the edge $v_1v_3$ can be drawn inside $f$ so that this triangle bounds
a face.
\item Hence, we can assume that $W$ does not contain such a cut vertex, and thus $W$ is a cycle.  Suppose that say $v_2v_i$,
where $i\in\{4,5,\ldots,k\}$, is an edge of $H$.  By planarity, $v_1$ and $v_3$ are not adjacent in $H$.  Moreover, the vertices $v_1$ and $v_3$ have no common
neighbors except for $v_2$ if $k\ge 5$ and except for $v_2$ and $v_4$ if $k=4$.  Hence, every triangle in $H+v_1v_3$ bounds a face.
\item Consequently, we can assume that $W$ is an induced cycle.  Suppose next that each pair of distinct non-consecutive vertices of $W$
has a common neighbor not in $W$.  By planarity, it follows that all vertices in $W$ have a common neighbor outside of $W$.  However,
since the outer face of $H$ has length greater than three, this implies that $H$ contains a non-facial triangle, which is a contradiction.
\item Therefore, there exist distinct non-consecutive vertices $u$ and $v$ of $W$ that do not have a common neighbor outside of $W$.
Then all triangles in $H+uv$ bound faces.
\end{itemize}

Thus, we can repeatedly add edges to $H$ while preserving its outer face and the invariant that every triangle bounds a face,
until all internal faces become triangles.
\end{proof}

Using this lemma, we can reformulate Conjecture~\ref{conj-main-strong} in terms of precoloring extension in plane graphs
(a vertex of a plane graph is \emph{internal} if it is not incident with the outer face, and \emph{external} otherwise);
as well as prove the equivalence with Conjecture~\ref{conj-main-strong-tria}.
\begin{theorem}\label{thm-equiv}
The following claims are equivalent for every positive integer~$n$:
\begin{itemize}
\item[\textnormal{(i)}] Let $G$ be a graph with $|V(G)|=n$ and without separating triangles. If $G$ has a drawing in the plane with one crossing such that all vertices not incident with the crossed edges
have degree at least five, then $G$ is 4-colorable.
\item[\textnormal{(ii)}] Let $H'$ be a plane graph with $|V(H')|=n$ such that every triangle in $H'$ bounds a face and all internal faces of $H'$ are triangles.
If $H'$ is rainbow-forbidding, then $H'$ has an internal vertex of degree at most four.
\item[\textnormal{(iii)}] Let $H$ be a plane graph with $|V(H)|=n$ such that every triangle in $H$ bounds a face.
If $H$ is rainbow-forbidding, then $H$ has an internal vertex of degree at most four.
\item[\textnormal{(iv)}] Every non-4-colorable graph $G\in \CC_0$ with $n$ vertices has an internal vertex of degree four.
\end{itemize}
\end{theorem}
\begin{proof}
We prove several implications:
\begin{description}
\item[(i)$\Rightarrow$(ii)] Let $H'$ be a rainbow-forbidding plane graph with $n$ vertices such that every triangle in $H'$ bounds a face and all internal faces of $H'$ are triangles.
Let $K=x_1x_2x_3x_4$ be the 4-cycle bounding the outer face of $H'$.  Note that the cycle $K$ is induced: Otherwise, since every triangle in $H'$ bounds a face,
$H'$ would consist of $K$ together with one additional edge, and $H'$ would not be rainbow-forbidding.
Consider the graph $G=H'+\{x_1x_3,x_2x_4\}$, drawn in the plane so that only the edges $x_1x_3$ and $x_2x_4$ cross.

We claim that $G$ does not contain any separating triangles.  Indeed, consider any triangle $T$ in $G$.  By symmetry, we can
assume that $T$ does not contain the edge $x_2x_4$, and thus $H'\cup T$ is a plane triangulation.  In particular, $H'\cup T$ is $3$-connected.
If $T$ is a separating triangle in $H'\cup T$, then the planarity and $3$-connectedness implies that $H'-V(T)$ has exactly
two components $K_1$ and $K_2$, each induced by all vertices of $H'$ drawn in one of the faces of $T$.  Since every triangle in $H'$ bounds
a face, it follows that $x_1x_3\in E(T)$, $x_2\in V(K_1)$ and $x_4\in V(K_2)$.  However, since the edge $x_2x_4$ of $G$ connects $K_1$ with $K_2$,
this implies that $G-V(T)$ has only one component,
and thus the triangle $T$ is not separating.

Since $H'$ is rainbow-forbidding, the graph $G$ is not 4-colorable, and thus by (i), there exists a vertex $v\in V(G)\setminus V(K)$
such that $4\ge \deg_G v=\deg_{H'} v$.  Hence, (i) implies (ii).

\item[(ii)$\Rightarrow$(iii)] Consider a rainbow-forbidding plane graph $H$ with $n$ vertices such that every triangle bounds a face, and let $K=x_1x_2x_3x_4$ be the 4-cycle bounding the outer face
of $H$.  By Lemma~\ref{lemma-cantri}, there exists a supergraph $H'$ of $H$ such that $V(H')=V(H)$, the outer face of $H'$ is bounded by $K$,
all internal faces of $H'$ are triangles, and every triangle in $H'$ bounds a face.  Clearly $H'$ is also rainbow-forbidding.
By (ii), $H'$ has an internal vertex $v$ of degree at most four.  Clearly $v$ is also an internal vertex of $H$ and its degree in $H$ is at most as large as in $H'$.
Hence, (ii) implies (iii).

\item[(iii)$\Rightarrow$(i)] Suppose now that $G$ is a graph such that $|V(G)|=n$, $G$ does not contain separating triangles, and $G$ has a drawing in the plane with one crossing such that all vertices not incident with the crossed edges
have degree at least five.  If the set $X$ of vertices incident with the crossed edges $e_1$ and $e_2$ of $G$ does not induce a clique in $G$, then $G$ is $4$-colorable by Lemma~\ref{lemma-topre4}.

Otherwise, the graph $H=G-\{e_1,e_2\}$ has a plane drawing where the 4-cycle on $X$ bounds the outer face.  All internal vertices of $H$
have the same degree as in $G$, i.e., at least five.  Moreover, every triangle in $H$ bounds a face, as otherwise it would be a separating triangle in $G$.
By (iii), $H$ is not rainbow-forbidding, and thus $H$ has a 4-coloring $\varphi$ whose restriction to $H[X]$ has the rainbow type.  Then $\varphi$ is also a proper 4-coloring of $G$.
Therefore, (iii) implies (i).

\item[(iv)$\Rightarrow$(ii)] Consider a rainbow-forbidding plane graph $H'$ with $n$ vertices such that every
triangle in $H'$ bounds a face and all internal faces of $H'$ are triangles.  Let $K=x_1x_2x_3x_4$ be the 4-cycle bounding the outer face of $H'$.
Note that $K$ is an induced cycle, as otherwise we would have $V(H')=V(K)$ and $H'$ would not be rainbow-forbidding.
For $i\in\{1,2\}$, let $G_i=H'+x_ix_{i+2}$, so that $G_i$ is a plane triangulation.  If $G_i$ is 4-connected, then $G_i+x_{3-i}x_{5-i}\in \CC_0$ has an internal
vertex of degree at most four by (iv), and so does $H'$.

Otherwise, for both $i\in\{1,2\}$, there must exist a separating triangle in $G_i$, necessarily
containing the edge $x_ix_{i+2}$ since all other triangles bound faces by the assumptions on $H'$.  By planarity, this is only possible if $H'$ contains a vertex $v$
adjacent to $x_1$, \ldots, $x_4$.  Since all triangles in $H'$ bound faces, we conclude that $\deg v=4$.  Therefore, (iv) implies (ii).
\end{description}
Finally, it is clear that (i) implies (iv).
\end{proof}

In particular, to verify Conjecture~\ref{conj-main-strong} for graphs with $n$ vertices, it suffices it verify the claim from Theorem~\ref{thm-equiv}(ii) for graphs with $n$ vertices.
For small $n$, this can be done by a computer-assisted enumeration, as we discuss in the following section.

Next, let us argue that Conjecture~\ref{conj-main-strong} is also equivalent to a claim about diagonal-forbidding graphs
(we say that a plane graph is \emph{diagonal-forbidding} if it is $x$-diagonal-forbidding for a vertex $x$ incident with its outer
face).  For an integer $n\ge 5$, let $D(n)$ be the following claim:
\begin{itemize}
\item[$D(n)$:] For every plane graph $H$ such that $5\le |V(H)|\le n$, all internal faces of $H$ are triangles, and
every triangle in $H$ bounds a face, if $H$ is diagonal-forbidding, then $H$ contains an internal vertex of degree at most four.
\end{itemize}
Let $R(n)$ be the analogous claim for rainbow-forbidding graphs.
\begin{theorem}\label{thm-equiv2}
For every integer $n\ge 5$, $R(n+1)$ implies $D(n)$ and $D(n+1)$ implies $R(n)$.
\end{theorem}
\begin{proof}
Let us prove that $R(n+1)$ implies $D(n)$.
Let $H$ be a plane graph with at least five but at most $n$ vertices such that all internal faces of $H$ are triangles and
every triangle in $H$ bounds a face, with the outer face bounded by a 4-cycle $K$.  Suppose that $H$ is $x$-diagonal-forbidding for a vertex $x\in V(K)$.
We need to argue that if $R(n+1)$ holds, then $H$ contains an internal vertex of degree at most four.

Note that $K$ is an induced cycle in $H$, since $H$ has at least five vertices and every triangle in $H$ bounds a face.
Furthermore, note that this implies that every vertex of $K$ has degree at least three in $H$.
If $|V(H)|=5$, then $H$ consists of $K$ and an internal vertex of degree four, and the conclusion holds (and also,
$H$ actually is not diagonal-forbidding).  

Hence, suppose that $|V(H)|\ge 6$.  If $\deg x=3$, then let $H'=H-x$.
Otherwise, let $H'$ be the graph obtained from $H$ by adding a vertex $x'$ drawn in the outer face of $H$ and adjacent
to $x$ and the two neighbors of $x$ in $K$.  Observe than in both cases, the outer face of $H'$ is bounded by a $4$-cycle, $H'$ is rainbow-forbidding, it has at
least five but at most $n+1$ vertices, all internal faces of $H'$ are triangles, and every triangle in $H'$ bounds a face.
By $R(n+1)$, there exists an internal vertex $v$ of $H'$ of degree at most four.  Note that $v\neq x$, since
if $x\in V(H')$, then $x$ has degree at least four in $H$ and at least five in $H'$.  Hence, $v$ is also
an internal vertex of $H$ of degree at most four.

The implication $D(n+1)\Rightarrow R(n)$ is proved analogously.
\end{proof}

Together with Theorem~\ref{thm-equiv}, this gives the following consequence.
\begin{corollary}\label{cor-combined}
The following claims are equivalent:
\begin{itemize}
\item Conjecture~\ref{conj-main-strong} holds.
\item $D(n)$ holds for every $n\ge 5$.
\item $R(n)$ holds for every $n\ge 5$.
\end{itemize}
\end{corollary}
Interestingly, the analogous claim for bichromatic-forbidding graphs is false, as we discuss at the end of Section~\ref{sec-enum-cand}.

The proof of Lemma~\ref{lemma-wheel} follows an idea similar to the proof of Theorem~\ref{thm-equiv}.
For a future use, we prove a somewhat stronger claim, which implies Lemma~\ref{lemma-wheel} by Lemma~\ref{lemma-topre4}.

\begin{lemma}\label{lemma-wheel2}
Let $H$ be a rainbow- or diagonal-forbidding plane graph such that every triangle in $H$ bounds a face and $|V(H)|\ge 5$.
If Conjecture~\ref{conj-main-strong} holds, then there exists an internal vertex $v\in V(H)$ of degree four such that all
incident faces have length three.
\end{lemma}
\begin{proof}
For any integer $k\ge 4$, let $F_k$ be a plane graph such that the outer face of $F_k$ is bounded by an induced cycle of length $k$,
all triangles in $F_k$ bound faces, all internal faces of $F_k$ have length three, and all internal vertices of $F_k$ have degree at least five.  For example, we can define $F_4$ to
be the graph obtained from the icosahedron by deleting an edge, and for $k\ge 5$, we can let $F_k$ be the $k$-wheel.
The claim of the lemma essentially follows by filling in each internal face $f$ of $H$ of length $|f|\ge 4$ by a copy of $F_{|f|}$
and applying $D(n)$ or $R(n)$ to the resulting $n$-vertex graph $H'$.  However, we need to be slightly more careful
in order to argue that $H'$ satisfies the assumptions of $D(n)$ or $R(n)$.

Suppose for a contradiction that $H$ is a counterexample to Lemma~\ref{lemma-wheel2} with the smallest number of vertices.
Let $K$ be the $4$-cycle bounding the outer face of $H$.  Note that the cycle $K$ is induced, since $|V(H)|\ge 5$ and every triangle in $H$ bounds a face.

Let us consider the case that $H$ contains a separating clique $Q$.
There exists a component $C$ of $H-V[Q]$ such that the subgraph $H_1=H[V(Q)\cup V(C)]$ satisfies $K\subseteq H_1$.
Let $H_2=H-V(C)$.  By the Four Color Theorem, $H_2$ has a 4-coloring $\varphi_2$.  Since $Q$ is a clique, for any 4-coloring $\varphi_1$ of $H_1$,
we can permute the colors in the coloring $\varphi_2$ so that it matches $\varphi_1$ on $Q$, and thus $\varphi_1\cup \varphi_2$ is a 4-coloring of $H$.
Since $H$ is rainbow- or diagonal-forbidding, it follows that $H_1$ also is rainbow- or diagonal-forbidding.
By the minimality of $H$, there exists an internal vertex $v\in V(H_1)$ of degree four such that all incident faces are bounded by triangles.
Since all triangles in $H$ bound faces, these triangles also bound faces in $H$, and thus $v$ has degree four in $H$ as well.
This is a contradiction, since $H$ is a counterexample.

Therefore, the graph $H$ does not contain any separating clique.  A similar argument shows that
every internal vertex of $H$ has degree at least four.
Since $H$ does not contain separating cliques, it is 2-connected and every face of $H$ is bounded by an induced cycle.
Let $H'$ be the plane graph obtained from $H$ by, for each internal face $f$ of length $|f|\ge 4$, adding a copy of $F_{|f|}$
and identifying the cycle bounding its outer face with the boundary cycle of $f$.  Since the cycles bounding the internal faces
of $H$ are induced, all triangles created by this modification bound faces, and thus every triangle in $H'$ bounds a face.
Since $H$ is rainbow- or diagonal-forbidding, $H'$ is rainbow- or diagonal-forbidding as well.  Moreover, all internal faces of $H'$ have length three.

Assuming that Conjecture~\ref{conj-main-strong} holds,
$R(|V(H')|)$ or $D(|V(H')|)$ applies and consequently $H'$ contains an internal vertex $v$ of degree at most four.  Since all internal vertices of the graphs
$F_k$ have degree at least five, $v$ is also an internal vertex of $H$.   All internal vertices of $H$ have degree at least four,
and thus $v$ has degree exactly four and all faces of $H$ incident with $v$ are triangles
(otherwise a copy of a graph $F_{|f|}$ would be glued into a face $f$ incident with $v$ in the construction of $H'$,
increasing the degree of $v$ to at least five).  Hence, the conclusion of the lemma holds for $v$.
\end{proof}

Finally, let us point out a useful fact about non-extendable precolorings in a plane graph $G$ with the outer face bounded by a 4-cycle $K=x_1x_2x_3x_4$.
If $K$ has a chord, say $x_1x_3\in E(G)$, then $G$ is both $x_1$-diagonal-forbidding and bichromatic-forbidding.  However, we claim that at most one type
of colorings can be forbidden when $K$ is an induced cycle.

To see this, let us define $\approx_r$, $\approx_1$, $\approx_2$, and $\approx_b$
to be the types of rainbow, $x_1$-diagonal, $x_2$-diagonal, and bichromatic $4$-colorings of $K$, let $Q$ be an
auxiliary graph formed by the 4-cycle with vertices $\approx_r$, $\approx_1$, $\approx_b$, and $\approx_2$ in order,
and let us note the following observation based on Kempe chains.

\begin{lemma}\label{obs-adjacent}
Let $G$ be a plane graph with the outer face bounded by a 4-cycle $K=x_1x_2x_3x_4$, let $\varphi$ be a 4-coloring of $G$,
and let $\approx$ be the type of the restriction of $\varphi$ to $K$.  There exists a $4$-coloring $\varphi'$ of $G$
such that the type of the restriction of $\varphi'$ to $K$ is adjacent to $\approx$ in $Q$.
\end{lemma}
\begin{proof}
This follows by the inspection of Kempe chains.  Let us illustrate the argument for a coloring $\varphi$ assigning
colors $1,2,1,3$ to vertices of $K$ in order, i.e., for the case that $\approx$ is $\approx_1$. 
Let $A$ be the subgraph of $G$ induced by vertices of colors $1$ and $4$, and let $B$ be the subgraph induced by the vertices
of colors $2$ and $3$.  If $A$ does not contain a path between $x_1$ and $x_3$, then we can exchange colors
$1$ and $4$ on the component of $A$ containing $x_3$, obtaining a 4-coloring of $G$ whose restriction to $K$ has type $\approx_r$.
If $B$ does not contain a path between $x_2$ and $x_4$, then we can exchange the colors $2$ and $3$ on the component of $B$ containing $x_4$, obtaining
a 4-coloring of $G$ whose restriction to $K$ has type $\approx_b$.  And, by planarity,
$G$ cannot at the same time contain a path between $x_1$ and $x_3$ in colors $1$ and $4$, as well as
a path between $x_2$ and $x_4$ in colors $2$ and $3$.
\end{proof}

Next, let us make some observations on the existence of $4$-colorings that follow from the Four Color Theorem.
\begin{observation}\label{obs-chordop}
Let $G$ be a plane graph with the outer face bounded by a 4-cycle $K=x_1x_2x_3x_4$.
Then either $G$ has a 4-coloring matching $\approx_r$, or two 4-colorings matching
$\approx_1$ and $\approx_2$.
\end{observation}
\begin{proof}
By the Four Color Theorem, the planar graph $G_1$ obtained from $G$ by adding the edge $x_2x_4$
has a 4-coloring $\varphi_1$, and the planar graph obtained from $G$ by adding the $x_1x_3$ has a 4-coloring $\varphi_2$.
Note that either at least one of $\varphi_1$ and $\varphi_2$ matches $\approx_r$, or
$\varphi_1$ matches $\approx_1$ and $\varphi_2$ matches $\approx_2$.
\end{proof}

\begin{observation}\label{obs-identop}
Let $G$ be a plane graph with the outer face bounded by an induced 4-cycle $K=x_1x_2x_3x_4$.
Then either $G$ has a 4-coloring matching $\approx_b$, or two 4-colorings matching
$\approx_1$ and $\approx_2$.
\end{observation}
\begin{proof}
This follows similarly to Observation~\ref{obs-chordop} by considering the planar graphs obtained from $G$ by identifying $x_1$ with $x_3$
and by by identifying $x_2$ with $x_4$.  Note that these graphs do not contain loops, since $K$ is an induced cycle.
\end{proof}

The desired claim is now just a simple combination of these observations.

\begin{corollary}\label{cor-allbutone}
Let $G$ be a plane graph with the outer face bounded by an induced 4-cycle $K=x_1x_2x_3x_4$.
Then $G$ has three 4-colorings whose restrictions to $K$ have different types.
\end{corollary}
\begin{proof}
Let $A$ be the set of vertices $\approx{}\in V(Q)$ such that $G$ has a 4-coloring matching $\approx$.
By Observations~\ref{obs-chordop} and \ref{obs-identop}, either $\approx_1,\approx_2{}\in A$, or $\approx_b,\approx_r{}\in A$.
In either case Observation~\ref{obs-adjacent} implies that $A$ contains at least one more type.
\end{proof}

\section{Generating candidates}\label{sec-enum-cand}

A \emph{canvas} is a plane graph with the outer face bounded by a cycle of length at least four
and all other faces of length three.
A \emph{candidate} is a canvas such that every triangle bounds a face and all internal vertices have
degree at least five.  A candidate is an \emph{$\ell$-candidate} if its outer face has length $\ell$.
By Theorem~\ref{thm-equiv}, in order to verify Conjecture~\ref{conj-main-strong} for graphs with at most $n$ vertices,
it suffices to enumerate the $4$-candidates with at most $n$ vertices and check that none of them is rainbow-forbidding.

In this section, we develop a way to enumerate $\ell$-candidates for any fixed $\ell\ge 4$.
More precisely, we are going to describe a system of reductions which turn each sufficiently large $\ell$-candidate into a smaller
one.  Thus, all $\ell$-candidates can be produced from small $\ell$-candidates by a sequence of the inverse operations to these reductions.

A cycle $K$ in a canvas is \emph{hollow} if no vertex is drawn in the open disk bounded by $K$ (but edges can be drawn in this disk).
A candidate is \emph{internally $k$-connected} if all cycles of length less than $k$ other than the one bounding the outer face
are hollow.  In particular, every candidate is internally $4$-connected.
For a canvas $G$ and a cycle $C$ in $G$, the subgraph $G'$ of $G$ drawn in the closed disk bounded by $C$ is a \emph{subcanvas}
of $G$.  When $G$ is a candidate, we also say that $G'$ is a \emph{subcandidate}, and it is an \emph{$m$-subcandidate} if $C$ has length $m$.
When we look for a reducible configuration, it is convenient to assume that the candidate is internally $5$-connected,
since then its reduction turns out to be guaranteed not to create non-facial triangles.
We can do so by focusing on the smallest $4$-subcandidate $H$ of the considered $\ell$-candidate $G$;
the minimality of $H$ clearly implies that $H$ is internally 5-connected.
In order for the reducible configuration found in $H$ to be valid in $G$, we need
to ensure that the degrees of external vertices of $H$ are not decreased below five when the reduction is performed
in $G$. To facilitate this, we are going to require that in the canvas resulting from the
reduction in $H$, all external vertices have degree at least four.  Thus, we say that a canvas $H$ is \emph{thick}
if its outer face has length at least five or all external vertices of $H$ have degree at least four.  Let us argue that
this condition is automatically satisfied by non-trivial internally $5$-connected candidates.

\begin{observation}\label{obs-thick}
Every internally $5$-connected $\ell$-candidate $H$ with at least one internal vertex is thick.
\end{observation}
\begin{proof}
By the definition of thickness, this is trivial if $\ell\ge 5$.  Hence, suppose that $\ell=4$, and let
$K=v_1v_2v_3v_4$ be the cycle bounding the outer face of $H$.  Suppose for a contradiction that $H$ is not thick,
and thus say $\deg v_1\le 3$.  Let $C$ be the cycle in $H$ formed by the path on the neighbors of $v_1$ and the path $v_2v_3v_4$;
we have $|V(C)|=\deg v_1+1\le 4$.  Observe that every vertex of $H$ except for $v_1$ is drawn in the closed disk bounded by $C$.
However, the candidate $H$ is internally $5$-connected, and thus the cycle $C$ is hollow.
Since $H$ has at least one internal vertex, we conclude that $\deg v_1=3$ and $H$ has exactly one internal vertex $v$.
However, then $\deg v\le |V(K)|=4$, contradicting the assumption that $H$ is a candidate.
\end{proof}

We are also going to need the following simple discharging argument.

\begin{lemma}\label{lemma-thicknobas}
Every thick $4$-candidate contains an internal vertex of degree five with no external neighbor.
\end{lemma}
\begin{proof}
Let $G$ be a thick $4$-candidate and let $K$ be the 4-cycle bounding the outer face of $G$.
Assign charge $c_0(v)=d-6$ to each vertex $v\in V(G)$ of degree $d$.  Since all faces of $G$ except for the outer one have length three
and $|V(K)|=4$, observe that $G$ has exactly $3|V(G)|-7$ edges.  Consequently,
$$\sum_{v\in V(G)} c_0(v)=2|E(G)| - 6|V(G)|=-14.$$
For each external vertex $v$ of $G$, we redistribute the charge as follows.
Let $u_1$, \ldots, $u_d$ be the neighbors of $v$ in order according to the drawing of $G$, where $u_1vu_d$ is a subpath of $K$;
since $G$ is thick, we have $d\ge 4$.  The vertex $v$ sends $1/2$ to $u_2$ and $u_{d-1}$ and $1$ to each of $u_3$, \ldots, $u_{d-2}$.
Let us remark that since all internal faces of $G$ have length three, the vertices $u_2$ and $u_{d-1}$ receive additional $1/2$ in charge from this
rule being applied at $u_1$ and $u_d$, respectively.

For each vertex $x\in V(G)$, let $c(x)$ be the charge of $x$ after the redistribution.
Since no charge is created or lost, we have $$\sum_{x\in V(G)} c(x)=\sum_{x\in V(G)} c_0(x)=-14.$$
Moreover, for each vertex $v\in V(K)$, we have
$$c(v)\ge c_0(v)-2\times 1/2 - (\deg v-4)\times 1=-3,$$
and for each internal vertex $x$ with a neighbor in $K$, we have
$$c(x)\ge c_0(x)+1=\deg x-5\ge 0.$$
Thus, the sum of charges of vertices in the closed neighborhood of $K$ is at least $-3|V(K)|=-12>-14$.
It follows that there exists a vertex $u\in V(G)\setminus V(K)$ with no external neighbor and with $c(u)=c_0(u)<0$, i.e., $\deg u<6$.
Since $G$ is a candidate, we have $\deg u=5$.
\end{proof}

\begin{figure}
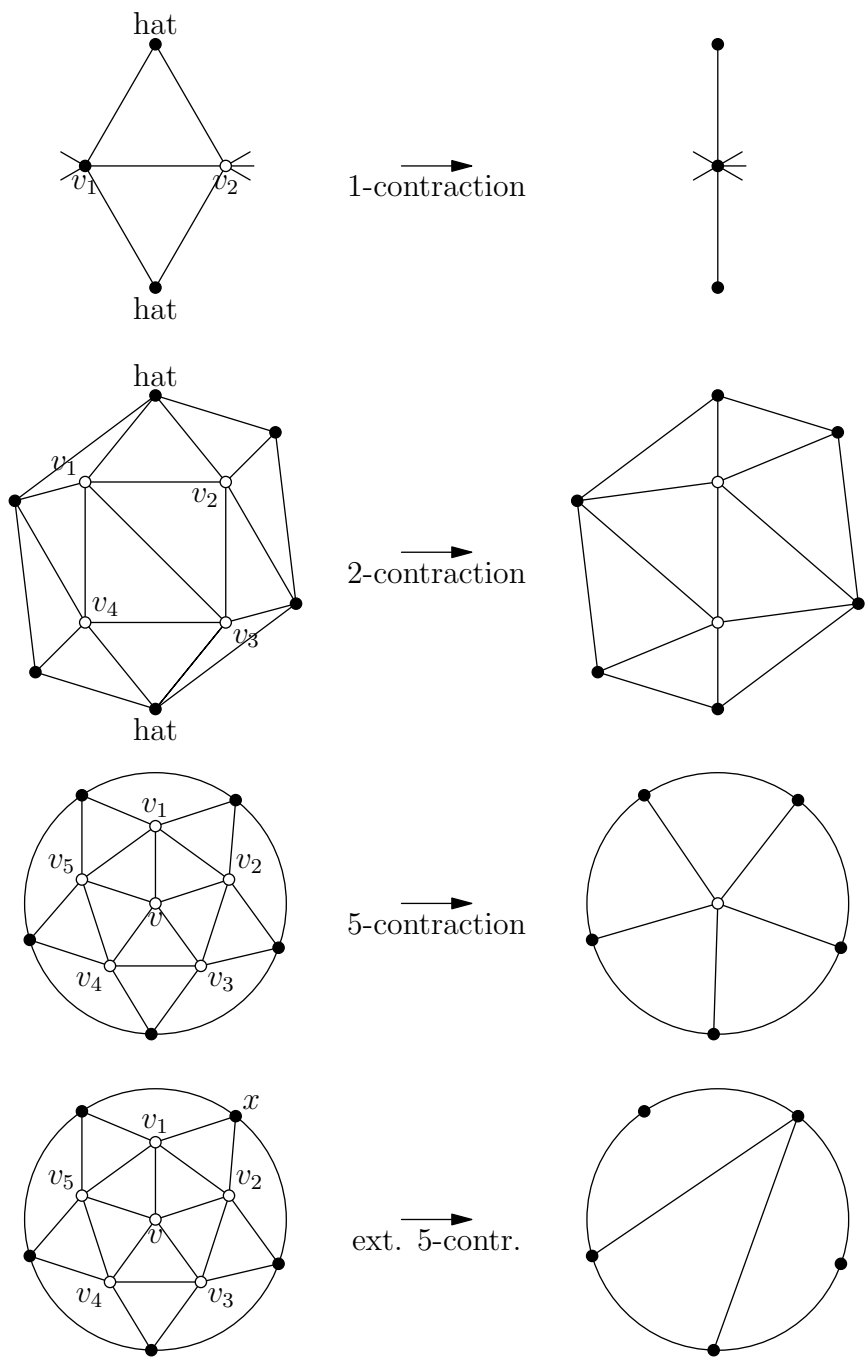

\begin{center}\asyinclude[width=0.9\textwidth]{fig-contrs.asy}\end{center}
\caption{$\star$-contractions (1)}\label{fig-contrs}
\end{figure}

\begin{figure}
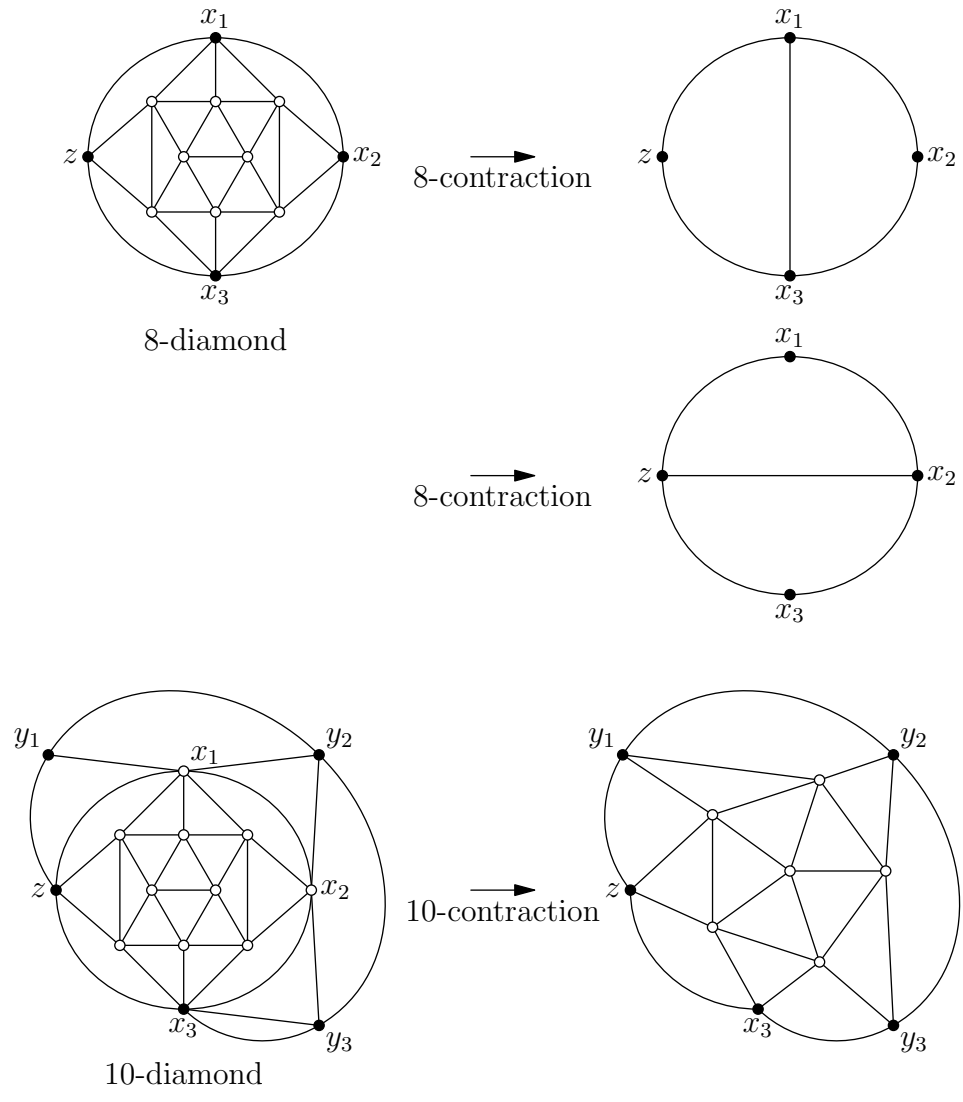

\begin{center}\asyinclude[width=\textwidth]{fig-contrs-rest.asy}\end{center}
\caption{$\star$-contractions (2)}\label{fig-contrs-rest}
\end{figure}

Let us now describe the reductions needed in the generation of candidates, see Figures~\ref{fig-contrs} and \ref{fig-contrs-rest} for an illustration.
After contraction of an edge, we always suppress the resulting 2-faces.
\begin{itemize}
\item The \emph{$1$-contraction} is a contraction of an edge $v_1v_2$, where $v_2$ is an internal vertex.
\item The \emph{$2$-contraction} is contraction of two edges $v_1v_2$ and $v_3v_4$, where $v_1v_2v_3v_4$ is a hollow $4$-cycle
(whose interior consists of two $3$-faces) and $v_1$, \ldots, $v_4$ are internal vertices of degree five.
\item The \emph{$5$-contraction} is a contraction of the five edges $v_1v$, \ldots, $v_5v$ incident to an internal vertex $v$ of degree five,
where $v_1$, \ldots, $v_5$ are also internal vertices of degree five.
\item The \emph{extended $5$-contraction} applies in the same situation.  Let $x\neq v$ be the common neighbor of $v_1$ and $v_2$.
The operation consists of the $5$-contraction of the edges incident with $v$ followed
by a $1$-contraction of the edge between the resulting vertex and $x$.
\item Let $K=x_1x_2x_3z$ be a 4-cycle such that the subgraph $H$ drawn in the closed disk bounded by $K$ is as depicted in Figure~\ref{fig-contrs-rest} on the top left;
in particular, all vertices of $V(H)\setminus V(K)$ are internal and have degree five.  We say that $H$ is an \emph{$8$-diamond}.
The \emph{$8$-contraction} is the operation of deleting the vertices of $V(H)\setminus V(K)$ and adding either the edge $x_1x_3$ or the edge $x_2z$.
\item Let $K=zy_1y_2y_3x_3$ be a 5-cycle such that the subgraph $H$ drawn in the closed disk bounded by $K$ consists of an $8$-diamond inside the 4-cycle $zx_1x_2x_3$
and the edges of the path $y_1x_1y_2x_2y_3x_3$, as depicted in Figure~\ref{fig-contrs-rest} on the bottom left.
in particular, all vertices of $V(H)\setminus (V(K)\cup \{x_1,x_2\})$ are internal and have degree five,
and the vertices $x_1$ and $x_2$ are internal and have degrees $7$ and $6$, respectively.
We say that $H$ is a \emph{$10$-diamond}.  The \emph{$10$-contraction} is the operation of deleting the vertices of $V(H)\setminus V(K)$
and replacing them by the $5$-wheel of six internal vertices of degree five.
\end{itemize}
Let us remark that this set of reductions is inspired by~\cite{brinkmann2005construction}.
By a \emph{$\star$-contraction}, we mean any of these operations.  We say that a $\star$-contraction is \emph{safe} if the result of the operation is a candidate,
i.e., it is internally $4$-connected and all internal vertices have degree at least five.
Since we are only working with simple graphs, this also implicitly contains the assumption that no parallel edges or loops
(except for those eliminated by suppression of $2$-faces) arise in the reduction.
A $\star$-contraction is \emph{strongly safe} if it is safe and the resulting candidate is thick.

For a $1$-contraction, the \emph{hats} are the two vertices distinct from $v_1$ and $v_2$
and incident with the faces whose boundary contains the edge $v_1v_2$.
For a $2$-contraction, the hats are the two vertices distinct from $v_1$, \ldots, $v_4$
and incident with the faces whose boundary contains the edge $v_1v_2$ or the edge $v_3v_4$.
The safety of $1$- and $2$-contractions can be described in terms of the degrees of the hats
as follows.

\begin{observation}\label{obs-12contr}
Let $G$ be a internally $5$-connected $\ell$-candidate. Consider a $1$- or $2$-contraction in $G$.
If neither of the hats of this $\star$-contraction is an internal vertex of degree five, then the $\star$-contraction is safe.
If additionally $\ell>4$ or neither of the hats is an external vertex of degree four, then the $\star$-contraction is strongly safe.
\end{observation}
\begin{proof}
Let the canvas $G'$ be created by performing the $1$- or $2$-contraction.
Since $G$ is internally $5$-connected, observe that this $\star$-contraction cannot create parallel edges or a non-facial triangle.
In the case of $1$-contraction, the vertex $v$ of $G'$ created by contracting the edge $v_1v_2$ has degree
$\deg(v_1)+\deg(v_2)-4>\deg v_1$, and it is internal if and only if $v_1$ is internal.  In the case of a $2$-contraction,
both vertices of $G'$ created by the contractions are internal and have degree five.  The degrees of vertices
other than hats are unchanged, and the degree of each hat decreases by one (note that the two hats are distinct,
since $G$ does not contain non-facial triangles).

Hence, if neither of the hats is an internal vertex of degree five, then every internal vertex of $G'$ has degree
at least five.  Consequently, $G'$ is a candidate, and thus the $\star$-contraction is safe.
If $\ell\ge 5$, then $G'$ is also automatically strongly safe.
For $\ell=4$, note that every external vertex has degree at least four in $G$ by Observation~\ref{obs-thick}.
If additionally neither of the hats is an external vertex of degree four, this implies that the candidate $G'$ is also thick,
and thus the $\star$-contraction is strongly safe.
\end{proof}

Let us now argue that with a few exceptions, a strongly safe $\star$-contraction can be performed in every internally $5$-connected candidate.
Let us start with a quite restrictive special case.

\begin{lemma}\label{lemma-int5}
Let $G$ be an internally $5$-connected $\ell$-candidate, not isomorphic to the $8$-diamond.
If $G$ contains an internal vertex $v$ of degree five such that all neighbors of $v$
are internal and have degree five, then $G$ admits a strongly safe $1$-contraction, $5$-contraction, or extended $5$-contraction.
\end{lemma}
\begin{proof}
Let $G'$ be the canvas obtained from $G$ obtained by the $5$-contraction of the edges incident with $v$,
and let $w$ denote the vertex created by the contraction.
Let $v_1$, \ldots, $v_5$ be the neighbors of $v$ in order.  For $i=1,\ldots,5$, let $x_i$ be the common neighbor of $v_i$ and $v_{i+1}$
(where $v_6=v_1$) distinct from $v$.  Note that $x_1$, \ldots, $x_5$ are the neighbors of $w$ in $G'$ in order.
For distinct $i,j\in\{1,\ldots,5\}$, we have $x_i\neq x_j$, as otherwise $G$ would contain a non-facial triangle.
In particular, $G'$ does not have any parallel edges.  Moreover, $x_i$ is adjacent to $x_j$ if and only if $|i-j|\equiv \pm 1\pmod 5$,
since $G$ is internally $5$-connected.  Therefore, $x_1\ldots x_5$ is an induced cycle in $G$ as well as in $G'$,
and consequently $G'$ does not contain any non-facial triangles.

\begin{itemize}
\item Let us first consider the case that $\ell>4$ or none of $x_1$, \ldots, $x_5$ is an external vertex of degree four in $G$.
Note that $\deg_{G'} x_i=\deg_{G} x_i-1$ for $i\in \{1,\ldots,5\}$, and thus this implies that the canvas $G'$
is thick.  Hence, if the $5$-contraction of the edges around $v$ is not strongly safe, then it is not safe, and one of the
vertices $x_1$, \ldots, $x_5$ is internal and has degree at most four in $G'$.

We can by symmetry assume that $x_2$ is an internal vertex of degree five in $G$, so that $x_2$ has degree four in $G'$.
Let $x'_2$ be the neighbor of $x_2$ distinct from $x_1$, $v_2$, $v_3$, and $x_3$; since all internal faces of $G'$ are triangles, $x'_2$ is
adjacent to $x_1$ and $x_3$.  If neither $x_1$ nor $x_3$ is an internal vertex of degree five in $G$, then by Observation~\ref{obs-12contr},
the $1$-contraction of the edge $x_2x'_2$ in $G$ is strongly safe.

If say $x_1$ were an internal vertex of degree five, then since all internal faces of $G$ are triangles, we would have $x'_2x_5\in E(G)$.
Since $G$ is internally 5-connected and the $4$-cycle $x'_2x_5x_4x_3$ is not hollow,
this cycle would bound the outer face of $G$.
But then $\ell=4$ and $x_4$ would be an external vertex of degree four, contradicting the assumptions of this case.

We conclude that in this case, $G$ admits a strongly safe $1$-contraction or $5$-contraction.

\item Hence, we can suppose that $\ell=4$ and say $x_2$ is an external vertex of degree four.
Then the outer face is bounded by the cycle $x_1x_2x_3z$ for some vertex $z$.  Note that $z\not\in \{x_4,x_5\}$, since otherwise
the cycle $zv_5v_4x_3$ or $zv_5v_1x_1$ would contradict the internal 5-connectivity of $G$.
Thus, $x_4$ and $x_5$ are internal vertices.

Let $G''$ be obtained from $G'$ by contracting the edge $wx_2$;
then the canvas $G''$ is obtained from $G$ by an extended $5$-contraction.  Note that $G''$ does not contain non-facial triangles,
since $x_1x_4,x_3x_5, x_1x_3\not\in E(G)$.  If $x_5z\in E(G)$, then since the $4$-cycle $x_5zx_3x_4$ is hollow
and $\deg_G(x_4)\ge 5$, we have also $x_4z\in E(G)$, and thus $G$ would be an $8$-diamond.
Hence, we can assume $x_4z\not\in E(G)$, and symmetrically that $x_5z\not\in E(G)$.
Hence, $zx_1x_5$ and $zx_3x_4$ are not $3$-faces of $G$, and since all internal faces of $G$ are triangles,
it follows that $\deg_G x_i\ge 6$ and $\deg_{G''} x_i\ge 4$ for $i\in \{1,3\}$.  Since $x_2x_4,x_2x_5\in E(G'')$,
we have $\deg_{G''} x_2=4$. Moreover, $\deg_{G''} z=\deg_G z\ge 4$ by Observation~\ref{obs-thick}.
If the extended $5$-contraction of $G$ to $G''$ is not strongly safe, it follows that $G''$ contains an internal vertex of degree at most four.
Clearly, this vertex must be $x_4$ or $x_5$.

Suppose that say $x_5$ has degree at most four in $G''$, and thus it has degree five in $G$.  Let $x'_5$ be the neighbor
of $x_5$ distinct from $x_1$, $v_1$, $v_5$, and $x_4$.  Since all internal faces of $G$ are triangles, we have $x'_5x_1,x'_5x_4\in E(G)$.
If $x'_5x_3\in E(G)$, then $x'_5x_1x_2x_3$ would be a non-hollow $4$-cycle in $G$.  It follows that $x'_5x_3\not\in E(G)$ and $x'_5x_4x_3$ is not a $3$-face
of $G$. Since all internal faces of $G$ are triangles, we conclude that $\deg_G x_4\ge 6$.  Observation~\ref{obs-12contr} then implies that the $1$-contraction
of the edge $x_5x'_5$ in $G$ is strongly safe.

We conclude that in this case, $G$ admits a strongly safe $1$-contraction, or $5$-contraction, or extended $5$-contraction.
\end{itemize}
\end{proof}

Next, we consider a much less restrictive special case.

\begin{lemma}\label{lemma-intc4}
Let $G$ be an internally $5$-connected $\ell$-candidate, not isomorphic to the $8$-diamond.
If $G$ contains a 4-cycle $C=v_1v_2v_3v_4$ of internal vertices of degree five,
then $G$ admits a strongly safe $\star$-contraction.
\end{lemma}
\begin{proof}
Since the candidate $G$ is internally $5$-connected, the open disk bounded by the 4-cycle $C$ contains exactly two faces;
by symmetry, we can assume that $v_1v_3$ is an edge drawn inside $C$.
For $i=1,\ldots,4$, let $x_i$ be the common neighbor of $v_i$ and $v_{i+1}$ (where $v_5=v_1$) distinct from $v_1$, \ldots, $v_4$.
Since $\deg v_1=\deg v_3=5$ and all internal faces of $G$ are triangles, we have $x_1x_4, x_2x_3\in E(G)$.
For $i\in \{2,4\}$, let $v'_i$ be the neighbor of $v_i$ different from $v_1$, \ldots, $v_4$, $x_1$, \ldots, $x_4$,
and similarly note that $v_i$ is adjacent to $x_{i-1}$ and $x_i$.  Moreover, $C'=x_1v'_2x_2x_3v'_4x_4$ is an induced cycle,
as otherwise $G$ would not be internally $5$-connected ($v'_2$ is not adjacent to $v'_4$, as otherwise by symmetry
between $v'_2x_2x_3v'_4$ and $v'_2x_1x_4v'_4$, we can assume that the former $4$-cycle does not bound the outer face,
and thus it must be hollow and $x_2$ or $x_3$ would be an internal vertex of degree four).

Let us first consider the case that either $\ell>4$ or none of $x_1$, \ldots, $x_4$ is an external vertex of degree four.
By Observation~\ref{obs-12contr}, if the $2$-contraction of $v_1v_2$ and $v_3v_4$ is not strongly safe, then $x_1$ or $x_3$
is an internal vertex of degree five; by symmetry, suppose this is the case for $x_1$.  If $x_4$ is an internal vertex of degree
five, then $G$ admits a strongly safe $\star$-contraction by Lemma~\ref{lemma-int5} applied to $v_1$ and its neighbors;
hence suppose this is not the case.  By Observation~\ref{obs-12contr}, if the $2$-contraction of $v_2v_3$ and $v_1v_4$ is not strongly safe,
then $x_2$ is an internal vertex of degree five.  If $x_3$ is an internal vertex of degree five, then $G$ admits a strongly safe $\star$-contraction by Lemma~\ref{lemma-int5} applied to $v_3$ and its neighbors.
Otherwise, Observation~\ref{obs-12contr} implies that the $1$-contraction of the edge $v_4v'_4$ is strongly safe.

Next, suppose that $\ell=4$ and at least one of the vertices $x_1$, \ldots, $x_4$, say $x_1$, is an external vertex of degree four.
Hence, the outer face of $G$ is bounded by the cycle $x_4x_1v'_2y$ for another vertex $y$. Since the cycle $C'$ is induced,
we have $y\not\in V(C')$, and thus the vertices $x_2$, $x_3$, and $v'_4$ are internal and $\deg x_4\ge 5$.
By Observation~\ref{obs-12contr}, if $\deg(x_3)>5$, then the $1$-contraction of $v_4v'_4$ is strongly safe,
and if $\deg(x_2)>5$, then the $2$-contraction of $v_1v_4$ and $v_2v_3$ is strongly safe.  Finally, if $\deg(x_2)=\deg(x_4)=5$,
then the claim follows from Lemma~\ref{lemma-int5} applied to $v_3$ and its neighbors.
\end{proof}

We can now consider the general case of internally 5-connected candidates.

\begin{lemma}\label{lemma-intde5}
Let $G$ be an internally $5$-connected $\ell$-candidate, not isomorphic to the $8$-diamond.
If $G$ contains an internal vertex, then $G$ admits a strongly safe $\star$-contraction.
\end{lemma}
\begin{proof}
Suppose first that $G$ has an internal vertex $v$ of degree five, and let $v_1$, \ldots, $v_5$ be the neighbors of $v$ in order.
If $\ell=4$, then by Observation~\ref{obs-thick} and Lemma~\ref{lemma-thicknobas}, we can additionally choose $v$ so that $v_1$, \ldots, $v_5$ are internal vertices.
If three consecutive neighbors of $v$ are internal vertices of degree five, then the claim follows by Lemma~\ref{lemma-intc4}.
Otherwise, we can by symmetry assume that neither $v_1$ nor $v_3$ is an internal vertex of degree five, and the $1$-contraction of the edge $vv_2$
is strongly safe by Observation~\ref{obs-12contr}.

Hence, we can assume that all internal vertices of $G$ have degree at least six.  By Observation~\ref{obs-thick} and Lemma~\ref{lemma-thicknobas},
it follows that $\ell>4$.  Hence, by Observation~\ref{obs-12contr}, the $1$-contraction of any edge incident with
an internal vertex of $G$ is strongly safe.
\end{proof}

Finally, let us get rid of the assumption of internal 5-connectivity.

\begin{lemma}\label{lemma-reduce}
Every $\ell$-candidate $G$ with at least one internal vertex admits a safe $\star$-contraction.
\end{lemma}
\begin{proof}
We can assume $G$ is not an $8$-diamond, as this 4-candidate admits a safe $8$-contraction.
Hence, by Lemma~\ref{lemma-intde5}, we can assume that $G$ is not internally $5$-connected.
Let $G'$ be a minimal $4$-subcandidate of $G$ with at least one internal vertex, and let
$C=x_1x_2x_3z$ be the cycle bounding the outer face of $G'$.  Note that $G'$ is internally $5$-connected.
Moreover, note that each vertex of $C$ which is not external in $G$ is incident with an edge not belonging to $E(G')$:
If say $z$ were internal and had no neighbors outside of $G'$, then since every internal face of $G$ is a triangle, the triangle
$x_1zx_3$ would bound a face drawn outside of $C$, and $x_1x_2x_3$ would be a non-facial triangle of $G$.
Hence, if the interior of $C$ is replaced by a thick 4-candidate, then every internal vertex belonging to $C$ has degree at least five
in the resulting graph.  We conclude that if $G'$ admits a strongly safe $\star$-contraction, then this $\star$-contraction is also safe in $G$.
Hence, by Lemma~\ref{lemma-intde5}, we can assume $G'$ is an $8$-diamond.

For $i\in \{1,3\}$, since $\deg_{G'} x_i=5$, the analysis from the previous paragraph shows that if $x_i$ is an internal vertex of $G$, then $\deg_G x_i>5$.
Suppose that $u\in V(C)$ is an internal vertex of $G$ with exactly one neighbor $u'$ outside of $G'$, and let $x$ and $y$ be the common neighbors of $u$ and $u'$ in $C$.
Observe that every 4-cycle containing the edge $uu'$ is hollow, since every path of length three from $u$ to $u'$ in $G$ passes through $x$ or $y$
and every triangle in $G$ bounds a face.  Thus, the $1$-contraction of the edge $uu'$ in $G$ is safe.
Hence, we can assume that in $G$, each internal vertex of $C$ has at least two neighbors outside of $G'$.

Let $G_1$ and $G_2$ be the $8$-contractions of $G$ obtained by replacing the interior of $C$ by the edges $x_1x_3$ and $x_2z$,
respectively.   Note that at most one of the graphs $G_1$ and $G_2$ contains a non-facial triangle, as otherwise all vertices of $C$
would have a common neighbor and $G$ would contain a non-facial triangle.

Suppose first that $G_1$ contains a non-facial triangle, and thus the vertices $x_1$ and $x_3$ have a common neighbor $u$ distinct from $x_2$ and $z$.
Note that $ux_2,uz\not\in E(G)$, as otherwise $G$ would either contain a non-facial triangle or $x_2$ or $z$ would be
an internal vertex with only one neighbor outside of $G'$.  Since all internal faces of $G$ are triangles, it follows that for $i\in \{1,3\}$, if the vertex $x_i$
is internal in $G$, then it has at least two neighbors outside of $G'$ in addition to $u$, and thus it has degree at least five in $G_2$.
Moreover, if $z$ or $x_2$ is an internal vertex of $G$, then recall that it has at least two neighbors outside of $G'$, and thus it also has degree at least five in $G_2$.
It follows that the $8$-contraction of $G$ to $G_2$ is safe.  

Similarly, if $G_2$ contains a non-facial triangle, then the $8$-contraction to $G_1$ is safe.
Hence, we can assume neither $G_1$ nor $G_2$ contains a non-facial triangle.  If the $8$-contraction neither to $G_1$ nor to $G_2$
is safe, then we can by symmetry assume that $x_1$ and $x_2$ are internal vertices of $G$ with exactly two neighbors outside of $G'$.
Let $y_1$ and $y_2$ be the neighbors of $x_1$ outside of $G'$, and let $y_2$ and $y_3$ be the neighbors of $x_2$ outside of $G'$.
Let $C'$ be the 5-cycle $y_1y_2y_3x_3z$, and let $G''$ be the subgraph of $G$ drawn in the closed disk bounded by $C'$.  Then $G''$ is
a $10$-diamond.  Since neither $x_3$ nor $z$ is an internal vertex of $G$ with
all neighbors in $G''$, the corresponding $10$-contraction is safe.
\end{proof}

Thus, we obtain the desired result on generation of candidates.

\begin{corollary}\label{cor-generate-cands}
For every $\ell$-candidate $G$, there exists an $\ell$-candidate $G'$ with no internal vertices
such that $G$ can be obtained from $G'$ by a sequence of inverses to $\star$-contractions, where all intermediate
canvases in this sequence are $\ell$-candidates.
\end{corollary}

Corollary~\ref{cor-generate-cands} can be used to generate all non-isomorphic $\ell$-candidates with at most $n$ vertices
in time $C_\ell(n)\cdot\text{poly}(n)$, were $C_\ell(n)$ is the number of such $\ell$-candidates.
The basic idea is as follows:
We start from the $\ell$-candidates with no internal vertices, i.e., from 2-connected outerplanar
graphs with $\ell$ vertices and all internal faces of length three (such graphs can be easily enumerated,
since the graphs obtained from their duals by deleting the vertex corresponding to the outer face are
exactly the subcubic trees with $\ell-2$ vertices).  We then process these $\ell$-candidates as well as all
$\ell$-candidates generated later.  For the currently processed $\ell$-candidate $G$, we perform all possible inverses to $\star$-contractions
to obtain larger $\ell$-candidates.  For each such $\ell$-candidate $G'$ with at most $n$ vertices, we first check whether we have not
seen it before (e.g., by storing canonical isomorphism-invariant codes of all encountered $\ell$-candidates in a hash table;
since we deal with plane graphs, such a canonical code can be computed in linear time), and if not, we schedule it for processing.

This procedure is guaranteed to generate all non-isomorphic $\ell$-candidates with at most $n$ vertices
by Corollary~\ref{cor-generate-cands}.  However, the test whether we have seen the currently considered $\ell$-candidate
before is somewhat time consuming and requires space linear in $C_\ell(n)$.  A more efficient approach is described in~\cite{McKaygener};
briefly, for each $\ell$-candidate, we canonically pick one way of obtaining it through the inverse to a $\star$-contraction,
then only process it when we reach it through this particular operation.  Let us mention that another advantage of this approach
is that it is easy to distribute, since each candidate is processed in isolation.

\begin{table}
\begin{center}
\begin{tabular}{|c|c|c|}
\hline
Vertices&4-candidates&bichromatic-forbidding\\
\hhline{|=|=|=|}
4	&1	&1\\
\hline
12	&1	&0\\
\hline
13	&3	&0\\
\hline
14	&11	&0\\
\hline
15	&37	&0\\
\hline
16	&134	&1\\
\hline
17	&470	&0\\
\hline
18	&1\,713	&0\\
\hline
19	&6\,150	&3\\
\hline
20	&22\,353	&10\\
\hline
21	&81\,158	&29\\
\hline
22	&296\,023	&122\\
\hline
23	&1\,081\,761	&490\\
\hline
24	&3\,965\,863	&1\,851\\
\hline
25	&14\,576\,016	&6\,915\\
\hline
26	&53\,728\,107	&25\,631\\
\hline
27	&198\,572\,799	&94\,323\\
\hline
28	&735\,894\,150	&346\,005\\
\hline
\end{tabular}
\end{center}
\caption{Numbers of all non-isomorphic 4-candidates with up to 28 vertices, and the numbers of non-isomorphic bichromatic forbidding
candidates with up to 28 vertices.  No rainbow- or diagonal-forbidding candidates except for the diamond were found.}\label{tab-numbers}
\end{table}

We implemented this algorithm and used it to generate all non-isomorphic 4-candidates with at most 28
vertices, testing more than $10^9$ graphs, see Table~\ref{tab-numbers} for exact statistics\footnote{More precisely,
two of the authors wrote their own implementations, one in C++, one in Python/SageMath; the latter
is slower and we only used it to double-check the claim up to 25 vertices.  The programs 
can be found at \url{https://lidicky.name/pub/mindeg5/}.}.
None of them is rainbow-forbidding, and only the one consisting of the 4-cycle bounding the outer face with a single
chord (the \emph{diamond}) is diagonal-forbidding.  This confirms Conjecture~\ref{conj-main-strong} for graphs with at most 28 vertices.

Interestingly, there turn out to be infinitely many bichromatic-forbidding 4-candidates.
Next, we describe a construction of such candidates. To justify its validity, it is
convenient to use the following result of Fisk~\cite{fisk1973combinatorial}.
\begin{lemma}\label{lemma-parity}
Let $G$ be a triangulation of an orientable surface, and let $\varphi$ be a
4-coloring of $G$.  For $i\in \{1,\ldots,4\}$, let $n_i$ be the number of
odd degree vertices of color $i$.  Then $n_1\equiv\ldots\equiv n_4\pmod 2$.
\end{lemma}

\begin{figure}
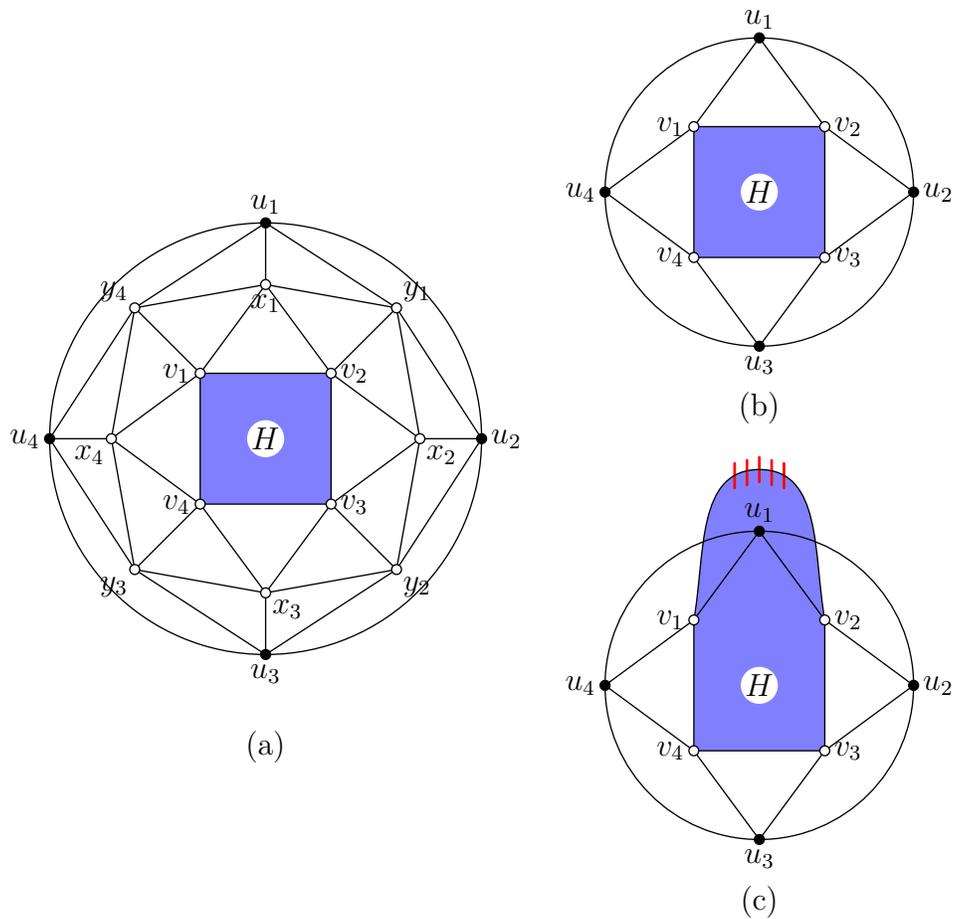

\begin{center}
\asyinclude[width=\textwidth]{fig-10ext.asy}
\end{center}
\caption{The constructions of bichromatic-forbidding 4-candidates.}\label{fig-10ext}
\end{figure}

The construction is as follows.

\begin{lemma}\label{lemma-10ext}
Let $H$ be a plane graph with the outer face bounded by the cycle $Q=v_1v_2v_3v_4$.
Let $G$ be the plane graph obtained as follows (see Figure~\ref{fig-10ext}(a)): We add a cycle $x_1y_1\ldots x_4y_4$
to the outer face of $H$, together with edges
$v_ix_i$, $v_ix_{i-1}$, and $v_iy_{i-1}$ for $i\in\{1,\ldots,4\}$,
where $x_0=x_4$ and $y_0=y_4$.  Then, we add the 4-cycle $K=u_1\ldots u_4$
bounding the outer face, and edges $u_iy_{i-1}$, $u_ix_i$, and $u_iy_i$
for $i\in \{1,\ldots, 4\}$.
The graph $G$ is bichromatic-forbidding if and only if the graph $H$ is bichromatic-forbidding.
Moreover, if $H$ is a 4-candidate, then $G$ is a 4-candidate.
\end{lemma}
\begin{proof}
Suppose first that $G$ is not bichromatic-forbidding, and thus $G$ has a 4-coloring $\varphi$ such
that $\varphi(u_1)=\varphi(u_3)=1$ and $\varphi(u_2)=\varphi(u_4)=2$.
We claim that $\varphi$ uses only two colors on $Q$, and thus the restriction of $\varphi$ to $H$ shows that $H$ is not bichromatic-forbidding.
For contradiction, suppose that $\varphi$ uses at least three colors on $Q$, and thus by symmetry,
we can assume that $\varphi(v_1)\neq \varphi(v_3)$.

Let $G'$ be the triangulation of the plane obtained from $G-(V(H)\setminus V(K))-(E(H)\setminus E(K))$
by adding the edge $v_1v_3$ and a vertex $z$ adjacent to $u_1$, \ldots, $u_4$,
and let $\varphi'$ be the 4-coloring of $G'$ obtained from the restriction of $\varphi$ to $V(G)\setminus (V(H)\setminus V(K))$
by letting $\varphi'(z)=3$.  Let $M=\{v_1,\ldots, v_4, x_1,\ldots,x_4, y_1,\ldots, y_4\}$ and for $c\in \{1,\ldots,4\}$, let $m_c$ be the
number of vertices in $M$ of color $c$.  Note that $M$ is covered by four triangles in $G'[M]$, namely $v_ix_{i-1}y_{i-1}$ for $i\in\{1,\ldots,4\}$,
and thus $m_1,\ldots,m_4\le 4$.  Moreover, since $\varphi'(u_1)=\varphi'(u_3)=1$, the color $1$ can only be used
on $M$ on the triangles $v_1v_4x_4$ and $v_2v_3x_2$, and thus $m_1\le 2$, and symmetrically $m_2\le 2$.  Since $m_1+\cdots+m_4=|M|=12$,
this is only possible if $m_1=m_2=2$ and $m_3=m_4=4$.

Note that the odd degree vertices of $G'$ are exactly those in $M'=M\setminus\{v_1,v_3\}$.  Each of the two colors that do not appear
on $v_1$ and $v_3$ is used an even number of times on $M'$ (twice or four times).
By Lemma~\ref{lemma-parity}, the remaining two colors $\varphi'(v_1)$ and $\varphi'(v_3)$ also have to be used even number of times on $M'$.  However, then
the color $\varphi'(v_1)$ would be used odd number of times on $M$, which is a contradiction, since $m_{\varphi'(v_1)}\in\{2,4\}$.

Conversely, suppose that $H$ is not bichromatic-forbidding, and thus $H$ has a 4-coloring $\psi$ such that $\psi(v_1)=\psi(v_3)=1$ and $\psi(v_2)=\psi(v_4)=2$.
Let us extend $\psi$ by letting $\psi(x_1)=\ldots=\psi(x_4)=3$, $\psi(y_1)=\ldots=\psi(y_4)=4$,
$\psi(u_1)=\psi(u_3)=1$ and $\psi(u_2)=\psi(u_4)=2$.  This gives a 4-coloring of $G$ which shows that $G$ is not bichromatic-forbidding.

Therefore, the graph $G$ is bichromatic-forbidding if and only if the graph $H$ is.
If $H$ is a $4$-candidate, then note that the construction of $G$ does not create non-facial triangles,
and clearly all vertices of $M$ have degree at least five in $G$; therefore, $G$ is a 4-candidate as well.
\end{proof}

According to Lemma~\ref{lemma-10ext}, we can start with the diamond (a trivial bichromatic-forbidding 4-candidate)
and iterate the construction described in the statement of the lemma to obtain larger and larger
bichromatic-forbidding 4-candidates.

A simpler construction is given in the following lemma.
\begin{lemma}\label{lemma-extproj}
Let $H$ be a 4-candidate with the outer face bounded by an induced cycle $Q=v_1v_2v_3v_4$.
Let $G$ be a canvas obtained from $H$ by adding a cycle $K=u_1u_2u_3u_4$ bounding the outer face
and the edges $u_iv_i$ and $u_iv_{i+1}$ for $i\in \{1,\ldots,4\}$, where $v_5=v_1$ (see Figure~\ref{fig-10ext}(b)).
Then $G$ is a 4-candidate, and it is bichromatic-forbidding if and only if $H$ is.
\end{lemma}
\begin{proof}
Since $Q$ is an induced cycle in $H$, all its vertices have degree at least three in $H$ and at least five in $G$.
It follows that $G$ is a candidate.

Suppose that $G$ is not bichromatic-forbidding, and thus $G$ has a 4-coloring $\varphi$ such that $\varphi(u_1)=\varphi(u_3)=1$
and $\varphi(u_2)=\varphi(u_4)=2$.  Then $\varphi$ can only use colors $3$ and $4$ on the 4-cycle $Q$,
and thus the restriction of $\varphi$ to $H$ shows that $H$ is not bichromatic-forbidding.

Conversely, if $H$ is not bichromatic-forbidding, then $H$ has a 4-coloring $\psi$ such that $\psi(v_1)=\psi(v_3)=3$
and $\psi(v_2)=\psi(v_4)=4$.  We can extend $\psi$ to a 4-coloring of $G$ by letting $\psi(u_1)=\psi(u_3)=1$
and $\psi(u_2)=\psi(u_4)=2$, thus showing that $G$ is not bichromatic-forbidding.

Therefore, $G$ is bichromatic-forbidding if and only if $H$ is.
\end{proof}

Since Lemma~\ref{lemma-extproj} requires $Q$ to be an induced cycle, it cannot be applied to the diamond directly.
However, it can be applied to the 4-candidates arising from Lemma~\ref{lemma-10ext}.
Note that Lemma~\ref{lemma-10ext} produces 4-candidates with all external vertices
of degree five and Lemma~\ref{lemma-extproj} ones with all external vertices of degree four.  There is also the following
variant of Lemma~\ref{lemma-extproj} producing 4-candidates with three external vertices of degree four and one of degree
at least seven.
\begin{lemma}\label{lemma-extprojprime}
Let $H$ be a 4-candidate with the outer face bounded by an induced cycle $Q=v_1v_2v_3v_4$,
and let $v_1v_2u_1$ be the triangle bounding the other face incident with the edge $v_1v_2$.
Let $G$ be a canvas obtained from $H-v_1v_2$ by adding a cycle $K=u_1u_2u_3u_4$ bounding the outer face
and the edges $u_iv_i$ and $u_iv_{i+1}$ for $i\in \{2,3,4\}$, where $v_5=v_1$ (see Figure~\ref{fig-10ext}(c)).
If $\deg_H v_1,\deg_H v_2\ge 5$, then $G$ is a 4-candidate.
Moreover, $G$ is bichromatic-forbidding if and only if $H$ is.
\end{lemma}
\begin{proof}
If $\deg_H v_1\ge 5$ and $\deg_H v_2\ge 5$, then all vertices of $Q$ have degree at least five in $G$.
From this, it follows that $G$ is a candidate.

Suppose that $G$ is not bichromatic-forbidding, and thus $G$ has a 4-coloring $\varphi$ such that $\varphi(u_1)=\varphi(u_3)=1$
and $\varphi(u_2)=\varphi(u_4)=2$.  Then $\varphi$ can only use colors $3$ and $4$ on the path $Q-v_1v_2$.
This implies that $\varphi(v_1)\neq \varphi(v_2)$, and thus the restriction of $\varphi$ to $V(H)$
is a proper 4-coloring of $H$.  This 4-coloring shows that $H$ is not bichromatic-forbidding.

Conversely, if $H$ is not bichromatic-forbidding, then $H$ has a 4-coloring $\psi$ such that $\psi(v_1)=\psi(v_3)=3$
and $\psi(v_2)=\psi(v_4)=4$.  By symmetry, we can assume that $\psi(u_1)=1$.  We can extend $\psi$ to a 4-coloring of
$G$ by letting $\psi(u_3)=1$ and $\psi(u_2)=\psi(u_4)=2$, thus showing that $G$ is not bichromatic-forbidding.

Therefore, $G$ is bichromatic-forbidding if and only if $H$ is.
\end{proof}

The aforementioned enumeration of $4$-candidates with at most $28$ vertices supports the following conjecture.
\begin{conjecture}
Every bichromatic-forbidding 4-candidate can be obtained from the diamond by a sequence of the operations
described in Lemmas~\ref{lemma-10ext}, \ref{lemma-extproj}, and \ref{lemma-extprojprime}.
\end{conjecture}

\section{Generating rainbow-forbidding weak $4$-candidates}\label{sec-gen4}

A canvas $G$ is a \emph{weak $\ell$-candidate} if its outer face is bounded by an $\ell$-cycle and all triangles in $G$ bound faces.  Note that this in particular implies that
all internal vertices of $G$ have degree at least four.  It is easy to see that to enumerate non-4-colorable graphs
in $\CC_0$, it suffices to enumerate the rainbow-forbidding weak $4$-candidates.
To do so, we fundamentally use the reduction discussed in the introduction,
deletion of an internal vertex $v$ of degree four followed by identification of two non-adjacent neighbors of $v$.
An issue here is that this might create a non-facial triangle, leading to a somewhat more involved argument presented in this section.
Let us remark that the same procedure can be used to generate diagonal-forbidding weak $4$-candidates; thus, we say that a plane graph $H$ with the
outer face bounded by a 4-cycle is \emph{restrictive} if it is rainbow- or diagonal-forbidding.
The \emph{kind} of $H$ is \texttt{r} if $H$ is rainbow-forbidding and \texttt{d} if $H$ is diagonal-forbidding.

Let $G$ be a restrictive weak $4$-candidate and let $P=uvw$ be a path in $G$,
where
\begin{itemize}
\item $v$ is an internal vertex of degree four and $u$, $x$, $w$, $y$ are its neighbors in order according to the plane drawing of $G$,
\item at least one of the vertices $u$ and $w$ is internal, and
\item $uw\not\in E(G)$ and $x$, $y$, and $v$ are the only common neighbors of $u$ and $w$.
\end{itemize}
Let $G/P$ be the canvas obtained by contracting the edges of $P$ and suppressing the 2-faces; by the assumptions, $G/P$ does not have loops
or parallel edges.  We say that a vertex $z$ of $G/P$ is \emph{triangle-isolated} if there exists a triangle $T$
in $G/P$ such that the open disk bounded by $T$ contains $z$.  For a non-negative integer $m$, we say that the path $P$ is \emph{$m$-contractible} if it satisfies the
conditions described above and the open disk bounded by any triangle in $G/P$ contains at most $m$ vertices.
Since every triangle in $G$ is facial, observe that all triangle-isolated vertices in $G/P$ are contained in open disks bounded by at most two triangles,
corresponding to 5-cycles in $G$ passing through $P$ and drawn on the opposite sides of $P$.
Hence, if the path $P$ is $m$-contractible, then $G/P$ contains at most $2m$ triangle-isolated vertices.

\begin{observation}\label{obs-reducopa}
Let $G$ be a weak $4$-candidate, let $P=uvw$ be an $m$-contractible path in $G$ for an integer $m\ge 0$, and let $F$ be the graph obtained from $G/P$ by deleting all triangle-isolated vertices.
If $G$ is restrictive, then $F$ is restrictive as well, and $F$ has the same kind as $G$.
\end{observation}
\begin{proof}
It suffices to show that every $4$-coloring $\varphi$ of $F$ extends to a $4$-coloring of $G$.  First, for every non-facial triangle $T$ in $G/P$,
we can extend $\varphi$ to the subgraph of $G/P$ drawn inside $T$: The subgraph has a $4$-coloring by the Four Color Theorem, and by permuting the colors,
we can modify it to match $\varphi$ on $T$.  Thus, $\varphi$ extends to a $4$-coloring of $G/P$.
Next, we extend $\varphi$ to $G$: We give $u$ and $w$ the color of the vertex resulting from the contraction of $P$, and then extend the coloring to $v$
greedily.
\end{proof}

It will be convenient to consider a variant of this operation: A \emph{bidiamond} consists of two adjacent internal vertices $u$ and $v$ of degree four.
Let $x$ and $y$ be the common neighbors of $u$ and $v$, and let $u'$ and $v'$ be the neighbors of $u$ and $v$, respectively, not in $\{u,v,x,y\}$.
If at least one of $x$ and $y$ is internal and there is no path of length at most three from $x$ to $y$ in $G-\{u',u,v,v'\}$, then we say that the bidiamond
is \emph{contractible}.  The \emph{contraction} of the bidiamond is the canvas obtained from $G-\{u,v\}$ by identifying $x$ with $y$ and
suppressing the 2-faces.  Observe that the resulting canvas is a restrictive weak $4$-candidate of the same kind as $G$.
According to the following lemma, even non-contractible bidiamonds are useful.

\begin{lemma}\label{lemma-redu-bidiamond}
Let $G$ be a restrictive triangulated weak $4$-candidate, and suppose that vertices $u$ and $v$ form a bidiamond in $G$.
Then either this bidiamond is contractible, or $G$ contains a $1$-contractible path.
\end{lemma}
\begin{proof}
Let $x$ and $y$ be the common neighbors of $u$ and $v$, and let $u'$ and $v'$ be the neighbors of $u$ and $v$, respectively, not in $\{u,v,x,y\}$.
We can assume that the path $u'uv$ is not $1$-contractible, and in particular there exists a path of length at most three
between $u'$ and $v$ in $G-\{u,x,y\}$.  The vertices $u'$ and $v'$ are non-adjacent, since $G$ does not contain separating triangles.
Therefore, $u'$ and $v'$ have a common neighbor $z\not\in \{x,y\}$.

By symmetry, we can assume that the open disk $\Lambda_1$ bounded by the 5-cycle $u'uvv'z$ contains $x$; and in particular, $x$ is an internal vertex.
Note that since $G$ does not contain non-facial triangles, $z$ is not adjacent to $y$.  
Let $z'$ be a common neighbor of $u'$ and $v'$ such that the open disk $\Lambda_2$ bounded by the 5-cycle $u'uvv'z'$ contains the face bounded by the triangle $uvy$;
note that $z'=y$ is possible.  Subject to these properties, choose the vertices $z$ and $z'$ so that the disks $\Lambda_1$ and $\Lambda_2$ are as large as possible.

We can assume that the bidiamond $\{u,v\}$ is not contractible, and thus there exists a path $R$ of length at most three from $x$ to $y$ in $G-\{u',u,v,v'\}$.
By planarity, $R$ intersects the path $u'zv'$, and since $zy\not\in E(G)$, we conclude that $R=xzy'y$ for a vertex $y'$.
Since all triangles is $G$ bound faces and $xz\in E(G)$, note that $x$ is the only vertex contained in $\Lambda_1$.
Moreover, observe that either $z'=y$ or $z'=y'$, and thus only the vertex $y$ can be contained in $\Lambda_2$.
By the maximality of $\Lambda_1$ and $\Lambda_2$, it follows that only $x$ and possibly $y$ is triangle-isolated in $G/u'uv$,
and thus the path $u'uv$ is $1$-contractible.
\end{proof}

We also need the following observations.
\begin{observation}\label{obs-pre3}
Let $G$ be a plane graph and let $P$ be a path of length at most two contained in the boundary of a face of $G$.
Then every precoloring of $G[V(P)]$ extends to a 4-coloring of $G$.
\end{observation}
\begin{proof}
We can assume that $P$ has length two, as otherwise we can extend $P$ by another vertex; let $P=v_1v_2v_3$.
Let $\psi$ be any precoloring of $G[V(P)]$.  Add the edge $v_1v_3$ to $G$, and if $\psi(v_1)=\psi(v_3)$, contract this
edge.  The resulting planar graph $G'$ is $4$-colorable by the Four Color Theorem, and we can permute the colors in this
4-coloring to match $\psi$.  Thus, the 4-coloring of $G'$ corresponds to a 4-coloring of $G$ extending $\psi$.
\end{proof}

\begin{observation}\label{obs-pre4}
Let $G$ be a plane graph and let $P$ be a path of length three contained in the boundary of a face of $G$.
Then every precoloring of $G[V(P)]$ giving the ends of $P$ the same color extends to a 4-coloring of $G$.
\end{observation}
\begin{proof}
Let $P=v_1v_2v_3v_4$ and let $\psi$ be any precoloring of $G[V(P)]$ such that $\psi(v_1)=\psi(v_4)$.
Since $\psi$ exists, we have $v_1v_4\not\in E(G)$.
Identify the vertices $v_1$ and $v_4$, turning $P$ into a triangle.  The resulting planar graph $G'$ is $4$-colorable by the Four Color Theorem, and we can permute the colors in this
4-coloring to match $\psi$.  Thus, the 4-coloring of $G'$ corresponds to a 4-coloring of $G$ extending $\psi$.
\end{proof}

We say that a canvas is \emph{trivial} if it has at most one internal vertex.  A restrictive weak $4$-candidate $G$ is \emph{deepest}
if it is non-trivial and no proper subcanvas of $G$ is a non-trivial restrictive weak $4$-candidate.
Lemma~\ref{lemma-redu-bidiamond} makes the following observation useful even when $G$ is a subcanvas of a larger weak $4$-candidate.
\begin{lemma}\label{lemma-noop}
Let $G$ be a deepest restrictive weak $4$-candidate with the outer face bounded by the cycle $K=x_1x_2x_3x_4$.
If $x_1$ and $x_3$ have a common neighbor $v\not\in \{x_2,x_4\}$, then $G$ contains a bidiamond.
\end{lemma}
\begin{proof}
For $i\in \{1,2\}$, let $K_i$ be the 4-cycle $x_{2i-1}x_2vx_4$ and let $G_i$ be the subcanvas of $G$ bounded by $K_i$.
Since $G$ is non-trivial and all triangles in $G$ bound faces, $G$ contains at most one of the edges $x_1v$ and $x_3v$.
We claim that both $G_1$ and $G_2$ are trivial, and thus $G$ contains a bidiamond.  For contradiction, suppose that say $G_1$ is non-trivial.
Clearly $G_1$ is a weak $4$-candidate, and since $G$ is deepest, $G_1$ is not restrictive.

Let $\varphi$ be a coloring of $K$ that does not extend to a 4-coloring of $G$, such that the type of $\varphi$ is rainbow or diagonal.
By Observation~\ref{obs-pre3}, the restriction of $\varphi$ to the path $x_2x_3x_4$ extends to a 4-coloring $\varphi_2$ of $G_2$.
Let $\psi_1$ be the 4-coloring of $K_1$ matching $\varphi$ on $x_2$, $x_1$, and $x_4$, and matching $\varphi_2$ on $v$.  Since $\varphi$ does not
extend to a 4-coloring of $G$, the coloring $\psi_1$ does not extend to a 4-coloring of $G_1$.  Since $G_1$ is not restrictive,
we can assume that $\varphi(x_1)=\psi_1(v)=1$ and $\varphi(x_2)=\varphi(x_4)=2$.  Since the type of $\varphi$ is not bichromatic,
we can assume that $\varphi(x_3)=3$.

Observe that Corollary~\ref{cor-allbutone} implies that there exists a color $c\in \{3,4\}$
such that $G_2$ has a 4-coloring $\varphi'_2$ matching $\varphi$ on the path $x_2x_3x_4$ and such that $\varphi'_2(v)=c$.
Since $G_1$ is not restrictive, it has has a 4-coloring $\varphi_1$ matching $\varphi$ on the path $x_2x_1x_4$ and such that $\varphi_1(v)=c$.
The combination of $\varphi'_2$ and $\varphi_1$ is a 4-coloring of $G$ extending $\varphi$, which is a contradiction.
\end{proof}

\begin{figure}
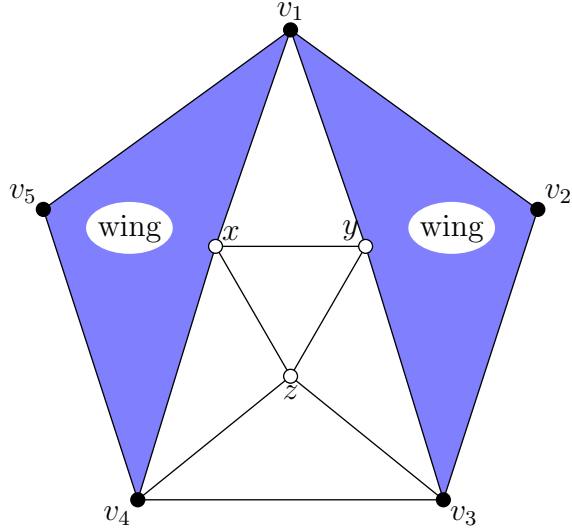

\begin{center}
\asyinclude[width=0.6\textwidth]{fig-trouble.asy}
\end{center}
\caption{A troublesome pentagon.}\label{fig-trouble}
\end{figure}

Next, we need to consider properties of two special configurations.
Let $G$ be a weak $4$-candidate.  A \emph{troublesome pentagon} in $G$ is a subcanvas $H$ of $G$ with the outer face bounded by
a 5-cycle $Q=v_1\ldots v_5$ containing a triangle $xyz$ of internal vertices such that $x$ is adjacent to $v_1$ and $v_4$,
$y$ is adjacent to $v_1$ and $v_3$, and $z$ is adjacent to $v_3$ and $v_4$; see Figure~\ref{fig-trouble}.  The vertex $z$ is the \emph{center} of the pentagon.
The subcanvases of $H$
bounded by the cycles $v_1xv_4v_5$ and $v_1v_2v_3y$ are the \emph{wings} of the troublesome pentagon.
We say that a vertex $w$ of $G$ is \emph{potentially useful} if there exists a restrictive induced subgraph $F$ of $G$
with the same outer face as $G$ such that $F-w$ is not restrictive.
\begin{lemma}\label{lemma-trouble}
Let $G$ be a restrictive weak $4$-candidate and let $H$ be a troublesome pentagon in $G$.
If neither of the wings of $H$ is restrictive, then the center of $H$ is not potentially useful.
\end{lemma}
\begin{proof}
Let $K$ be the 4-cycle bounding the outer face of $G$.
Let $Q=v_1\ldots v_5$ and $xyz$ be as in the definition of a troublesome pentagon and let $H_1$ and $H_2$
be the wings, where $v_2\in V(H_2)$.  Suppose for a contradiction that $z$ is potentially useful, and thus there exists a 
restrictive induced subgraph $F$ of $G$ with the outer face bounded by $K$ such that $F-z$ is not restrictive.
Let $\psi$ be a precoloring of $K$ (of rainbow or diagonal type) that does not extend to a 4-coloring of $F$,
but does extend to a 4-coloring of $F-z$.

Let $F'=(F-z)\cup Q$. Note that $F'$ is obtained from $F-z$ by adding the vertices of $V(Q)\setminus V(F)$
and the incident edges of $Q$, and thus $F-z$ is an induced subgraph of $F'$ and all vertices of $V(F')\setminus V(F)$
have degree two in $F'$.  Hence, $\psi$ also extends to a 4-coloring of $F'$.
Let $F''$ be the subgraph of $F'$ obtained by deleting the vertices and edges drawn in the open disk bounded by $Q$.

Consider any 4-coloring $\varphi$ of $F''$ extending $\psi$.
Since $\psi$ does not extend to a 4-coloring of $F$, the coloring $\varphi$ does not extend to a 4-coloring of $H$.
We claim that this implies the following:
\begin{itemize}
\item[$(\star)$] Every 4-coloring $\varphi$ of $F''$ extending $\psi$ satisfies either
$\varphi(v_1)=\varphi(v_4)$ and $\varphi(v_3)=\varphi(v_5)$, or $\varphi(v_1)=\varphi(v_3)$ and $\varphi(v_2)=\varphi(v_4)$.
\end{itemize}
Indeed, if $\varphi(v_1)\not\in\{\varphi(v_3),\varphi(v_4)\}$, then we could give $x$ the color $\varphi(v_3)$, $y$ the color $\varphi(v_4)$,
extend the coloring to $z$, and extend the coloring to $H_1$ and $H_2$ using the assumption that they are not restrictive, which is a contradiction.
Hence, suppose that $\varphi(v_1)=\varphi(v_4)$ (the case $\varphi(v_1)=\varphi(v_3)$ is symmetric).  Then we can color $x$ by $\varphi(v_3)$
and extend the 4-coloring to $z$, $y$, and $H_2$.  Since $\varphi$ does not extend to $H$, this coloring does not extend to $H_1$.
It follows that the subcanvas $H_1$ is bichromatic-forbidding and the coloring of its boundary 4-cycle given by $\varphi$ and the assignment of
the color $\varphi(v_3)$ to $x$ uses only two colors, i.e., $\varphi(v_5)=\varphi(v_3)$.  This confirms the conclusion of $(\star)$.

Since $\psi$ extends to a 4-coloring of $F-z$, it also extends to a 4-coloring $\varphi$ of $F''$.
By $(\star)$, we can assume that $\varphi(v_1)=\varphi(v_4)=1$ and $\varphi(v_3)=\varphi(v_5)=2$.
For distinct colors $a,b\in\{1,\ldots,4\}$, let $C_{a,b}$ be the subgraph of $F''$ induced by the vertices whose colors according to $\varphi$
are $a$ or $b$.

If $v_1$ and $v_4$ are in the same component of $C_{1,3}$, then by planarity, there exists $i\in \{3,5\}$
such that the component $C_1$ of $C_{2,4}$ containing $v_i$ does not contain $v_{8-i}$ and does not intersect $K$.  We can
exchange the colors $2$ and $4$ on $C_1$ to obtain a 4-coloring of $F''$ violating $(\star)$, which is a contradiction.
Therefore, $v_1$ and $v_4$ are not in the same component of $C_{1,3}$, and similarly, they are not in the same component of $C_{1,4}$.

If there exists $i\in \{1,4\}$ such that the component $C_2$ of $C_{1,3}$ or $C_{1,4}$ containing $v_i$ is disjoint from $K$,
we can exchange the colors on $C_2$, again obtaining a 4-coloring of $F''$ violating $(\star)$.  Thus, we can assume that all such components
intersect $K$.  Since $v_1$ and $v_4$ are not in the same component of $C_{1,3}$ or $C_{1,4}$, for both $a\in \{3,4\}$,
there exist distinct non-adjacent vertices $s_a, t_a\in V(K)$ such that $\psi(s_a),\psi(t_a)\in \{1,a\}$ and all
other vertices of $K$ have colors different from $1$ and $a$.

If $\{s_3,t_3\}=\{s_4,t_4\}$, this implies that $\psi(s_3)=\psi(t_3)=1$ and $\psi$ assigns the color $2$ to the remaining two vertices of $K$.
If $\{s_3,t_3\}\neq\{s_4,t_4\}$, then $\psi(s_3)=\psi(t_3)=3$ and $\psi(s_4)=\psi(t_4)=4$.  In either case, the type of $\psi$ were bichromatic,
which is a contradiction.
\end{proof}

\begin{figure}
\begin{center}
\asyinclude[width=0.9\textwidth]{fig-expansion.asy}
\end{center}
\caption{Configurations from Theorem~\ref{thm-redu-weak-cand} and the corresponding expansions.}\label{fig-expansion}
\end{figure}

Let $G$ be a weak $4$-candidate.  An \emph{amoeba} in $G$ is a subcanvas $H$ of $G$ with the outer face bounded by
a 4-cycle $Q=v_1\ldots v_4$, where $H$ contains a triangle $wyz$ of internal vertices,
$v_1$ is adjacent to $w$ and $y$, $v_2$ is adjacent to $w$ and $z$, and $v_3$ is adjacent to $z$, and moreover
$y$ and $v_3$ have a common neighbor $x\neq z$, where either $x$ is an internal vertex of $H$ or $x=v_4$; see bottom right of Figure~\ref{fig-expansion}.
The subcanvas of $G$ bounded by the 4-cycle $C=xyzv_3$ is the \emph{eye} of the amoeba, and
the subgraph of $G$ resulting from deleting all vertices and edges drawn in the open disk bounded by $C$
is said to be obtained from $G$ by \emph{blinding} the amoeba.  The amoeba is \emph{proper} if its eye has at least one internal vertex.
The eye of $H$ is irrelevant for the restrictive precoloring extension in $G$, in the following sense.

\begin{lemma}\label{lemma-amoeba}
Let $G$ be a restrictive weak $4$-candidate, let $\psi$ be a precoloring of the 4-cycle $K$ bounding the outer face of $G$
that does not extend to a precoloring of $G$, and let $H$ be an amoeba in $G$.  Then $\psi$ does not extend to a 4-coloring
of the subgraph $G'$ of $G$ obtained by blinding the amoeba $H$.
\end{lemma}
\begin{proof}
Let $Q=v_1\ldots v_4$, $w$, $x$, $y$, and $z$ be as in the definition of an amoeba.  Suppose for a contradiction that $\psi$
extends to a 4-coloring $\varphi$ of $G'$.

If $\varphi(v_1)\neq \varphi(v_3)$, then we let $\varphi_0$ be the restriction of $\varphi$ to the vertices of $G$ which are not internal
in $H$ and we show that $\varphi_0$ can be extended to a 4-coloring of $G$, obtaining a contradiction: We extend $\varphi_0$ to $z$ by letting
$\varphi_0(z)=\varphi(v_1)$.  By Observation~\ref{obs-pre4}, $\varphi_0$ further extends to the subcanvas of $G$ bounded by the 5-cycle
$zv_3v_4v_1y$.  Since $\varphi_0(v_1)=\varphi_0(z)$, we can finally extend this 4-coloring to $w$ as well.

Therefore, we have $\varphi(v_1)=\varphi(v_3)$, and thus $\varphi(z)\neq \varphi(v_1)$, and consequently $\varphi(y)=\varphi(v_2)$.
Let $\varphi_1$ be the restriction of $\varphi$ obtained by uncoloring the vertices $z$ and $w$.  By Observation~\ref{obs-pre3},
$\varphi_1$ can be extended to the eye of $H$, and since $\varphi_1(y)=\varphi_1(v_2)$, we can finally extend $\varphi_1$ to $w$ as well.
This is again a contradiction.
\end{proof}

Finally, we are ready to state the result showing that every non-trivial restrictive weak $4$-candidate
contains a reducible configuration.

\begin{theorem}\label{thm-redu-weak-cand}
If Conjecture~\ref{conj-main-strong} holds, then
every non-trivial restrictive weak $4$-candidate $G$ contains a contractible bidiamond, a $1$-contractible path, or a proper amoeba.
\end{theorem}
\begin{proof}
By Lemma~\ref{lemma-redu-bidiamond}, we can assume that $G$ does not contain a bidiamond.
Let $H$ be a minimal non-trivial restrictive subcanvas of $G$; note that $H$ is deepest.
Let $Q$ be the 4-cycle bounding the outer face of $H$.  Since $H$ is non-trivial and
all triangles in $G$ are facial, the cycle $Q$ is induced.  By Lemmas~\ref{lemma-redu-bidiamond} and \ref{lemma-noop},
we can assume that no internal vertex of $H$ has non-adjacent neighbors in $Q$.  

Suppose now that there exists an induced path $S=v_1vv_3w$ in $H$ such that $v_1,w\in V(Q)$, $v,v_3\not\in V(Q)$, and $v$ has degree four.
Let $v_1v_2v_3v_4$ be the cycle formed by the neighbors of $v$.  Note that $\{v_2,v_4\}\not\subseteq V(Q)$, since $v$ does not
have two non-adjacent neighbors in $Q$; by symmetry, we can assume that $v_4\not\in V(Q)$.
Let $Q=v_1u_1wu_2$, where $v_4$ is drawn in the open disk bounded by the $5$-cycle $v_1vv_3wu_2$.  Since $v_4$ does not have
two non-adjacent neighbors in $Q$, it is not adjacent to $w$.  Similarly, if $v_2\neq u_1$, then $v_2w\not\in E(G)$.
We can assume that the path $v_2vv_4$ is not $1$-contractible in $G$, and thus there exists an induced path $R$ of length at most three in $G-\{v,v_1,v_3\}$
between $v_2$ and $v_4$.  This is not possible if $v_2\neq u_1$, since then $v_4w,v_2w,u_1u_2\not\in E(G)$.  Therefore, $v_2=u_1$.

Suppose that $v_4$ is adjacent to $u_2$.  Note that $v_3$ is not adjacent to $v_2$, as otherwise $v$ and $v_4$ would form a bidiamond.
Consider any $4$-coloring $\psi$ of $Q$ whose type is rainbow or diagonal.
By symmetry, we can assume that $\psi(v_1)\neq\psi(w)$.  We can color $v_3$ by the color $\psi(v_1)$ and by Observation~\ref{obs-pre3}, extend the
4-coloring to the subcanvas drawn in the closed disk bounded by the 4-cycle $wv_3v_4u_2$.  Finally, we can greedily extend the coloring to $v$,
since two of its neighbors ($v_1$ and $v_3$) have the same color.  Therefore, any such 4-coloring of $Q$ extends to a 4-coloring $H$,
which is a contradiction since $H$ is restrictive.

Hence, we can assume that $v_4u_2\not\in E(G)$.
Let us now continue the discussion of the path $R$ showing that the path $v_2vv_4$ is not $1$-contractible.
Since $v_4$ is also not adjacent to $w$, the neighbor $x$ of $v_4$ in $R$ is contained in $V(H)\setminus V(Q)$.
Since $Q$ is an induced cycle, it follows that $R=v_4xwv_2$.  Therefore, $H$ is an amoeba.  Since the path $R$ shows that $v_2vv_4$ is not $1$-contractible,
the open disk bounded by the 5-cycle $v_2vv_4xw$ must contain at least two vertices, and thus the eye of the amoeba $H$ has at least one internal vertex.
It follows that the amoeba $H$ is proper, and thus the conclusion of the Lemma holds.
In conclusion,
\begin{itemize}
\item[($\star$)] there is no induced path in $H$ of length at most three with ends in $Q$ containing a vertex in $V(H)\setminus V(Q)$ of degree four.
In particular, every induced $(\le\!5)$-cycle $C$ in $G$ such that $C$ contains a vertex in $V(H)\setminus V(Q)$ of degree four
satisfies $C\subseteq H$.
\end{itemize}

Recall that the canvas $H$ is restrictive, and let $H'$ be a minimal induced subgraph of $H$ such that $Q\subseteq H'$ and $H'$ is restrictive.
In particular, all internal vertices of $H'$ are potentially useful in $H$.  Since $H$ is non-trivial and every triangle in $H$ bounds
a face, observe that $H'$ is non-trivial as well.  By Lemma~\ref{lemma-wheel2},  there exists an internal vertex $v\in V(H')$ of degree four such that all
incident faces of $H'$ have length three.  Thus, $v$ also has degree four in $H$.

Let $v_1v_2v_3v_4$ be the cycle formed by the neighbors of $v$;
since each triangle in $G$ bounds a face, this cycle is induced.  As we have observed, $v$ does not have two non-adjacent neighbors in $Q$,
and thus $\{v_1,v_3\}\not\subseteq V(Q)$ and $\{v_2,v_4\}\not\subseteq V(Q)$.  In particular, $v_1$ or $v_3$ is an internal vertex, and $v_2$ or $v_4$ is an internal vertex of $H$.

We can assume that the path $P_1=v_1vv_3$ is not $1$-contractible, as otherwise the conclusion of the Lemma holds.
Thus, $G-\{v,v_2,v_4\}$ contains an induced path $R_1$ of length at most three from $v_1$ to $v_3$ such that either $R_1$ has length two,
or $R_1$ has length three and the open disk bounded by the triangle $T_1$ in $G/P_1$ corresponding to $R_1$ contains at least two vertices.
By symmetry, we can assume that $v_4$ is contained in the open disk $\Lambda_1$ bounded by the cycle $P_1\cup R_1$ in $G$.

Similarly, we can assume that the path $P_2=v_2vv_4$ is not $1$-contractible, and thus $G-\{v,v_1,v_3\}$ contains an induced path $R_2$ of length
at most three from $v_2$ to $v_4$ such that either $R_2$ has length two, or $R_2$ has length three and the open disk bounded by the triangle $T_2$ in
$G/P_2$ corresponding to $R_2$ contains at least two vertices.  By symmetry, we can assume that $v_3$ is contained in the open disk $\Lambda_2$
bounded by the cycle $P_2\cup R_2$ in $G$.

By ($\star$), we have $R_1\cup R_2\subseteq H$.
By planarity, the paths $R_1$ and $R_2$ intersect in (at least) a vertex $z\in V(H)\setminus\{v, v_1, \ldots, v_4\}$.
Since each triangle in $G$ bounds a face, $z$ is not adjacent to both $v_1$ in $R_1$ and $v_2$ in $R_2$.  By symmetry,
we can assume that $z$ is not adjacent to $v_1$ in $R_1$, and thus the path $R_1$ has length three and $R_1=v_1z_1zv_3$ for a vertex $z_1$.

Suppose now that $z$ is adjacent to $v_2$ in $R_2$.  If $R_2$ has length two, then $z$ is also adjacent to $v_4$, and since all triangles
in $G$ bound faces, it follows that $\deg v_3=4$; however, then $v$ and $v_3$ would form a bidiamond.  Therefore, $R_2$ has length three,
and thus the subcanvas $H''$ of $H$ drawn in the closed disk bounded by the 4-cycle $v_1z_1zv_2$ is an amoeba. Moreover,
since the open disk bounded by $T_2$ contains at least two vertices, the eye of this amoeba contains at least one internal vertex.
Therefore, the amoeba $H''$ is proper, and thus the conclusion of the Lemma holds.

Hence, we can assume that $z$ is not adjacent to $v_2$ in $R_2$, and thus the path $R_2$ has length three and $R_2=v_2z_2zv_4$ for a vertex $z_2$.
If $v_1z_2\in E(G)$, then note that $v_1v_2z_2$ is a triangle bounding a face.  Observe that the subcanvas of $G$ drawn in the closed disk bounded by the
4-cycle $zv_4v_1z_2$ is an amoeba. Moreover, since the open disk bounded by $T_2$ contains at least two vertices, this amoeba is proper, and thus the conclusion of the Lemma holds.
Therefore, we can assume that $v_1z_2\not\in E(G)$, and symmetrically $v_2z_1\not\in E(G)$.

Let $F$ be the subcanvas of $H$ drawn in the closed disk bounded by the 5-cycle $v_1v_2z_2zz_1$.
Then $F$ is a troublesome pentagon with center $v$.  Since $v$ is potentially useful in $H$,
a wing of $F$ is restrictive.  Thus, by symmetry, we can assume that the subcanvas $H_1$ of
$H$ drawn in the closed disk bounded by the cycle $Q_1=v_1v_4zz_1$ is restrictive.

By the minimality of $H$, the subcanvas $H_1$ is trivial.  Since the open disk bounded by $T_1$ contains
at least two vertices, $H_1$ has exactly one internal vertex $y$ and $\deg y=4$.
Since the open disk bounded by $T_2$ contains at least two vertices, we have $v_3z_2\not\in E(G)$.
Since $v_2z_1\not\in E(G)$, it follows that the distance between $z_1$ and $v_4$ in $G-\{v_1,y,z\}$
is greater than three, and thus the path $v_4yz_1$ is $0$-contractible.  Therefore, the conclusion
of the theorem holds.
\end{proof}

Let $F$ be a weak $4$-candidate.  We say that a weak $4$-candidate $G$ is an \emph{expansion} of $F$ if $G$ is obtained from $F$ by one of the following operations,
see Figure~\ref{fig-expansion}:
\begin{itemize}
\item Decontraction of a bidiamond:  Choose a path $u'pv'$ in $F$.  Double the edges $u'p$ and $pv'$, letting $f_1$ and $f_2$
be the resulting $2$-faces.  Split $p$ into two vertices $x$ and $y$ along a simple curve from $f_1$ to $f_2$,
and add new vertices $u$ and $v$ and edges $ux$, $u'u$, $uy$, $uv$, $vx$, $vv'$, and $vy$ to the resulting face.
\item Decontraction of a $1$-contractible path:  Choose a vertex $p$ of $F$.  For $i\in \{1,2\}$,
either
\begin{itemize}
\item choose an edge $x_i=py_i$ incident with $p$, double it and let $f_i$ be the resulting $2$-face, or
\item choose a $3$-face $x_i$ incident with $p$, add a vertex $y_i$ of degree four inside $x_i$ adjacent to all the vertices of the boundary of $x_i$ and
with the edge $x_ip$ having multiplicity two, and let $f_i$ be the $2$-face bounded by this double edge,
\end{itemize}
so that $x_i\neq x_{3-i}$, and $x_i$ is not incident with $x_{3-i}$ in the case that $x_i$ is an edge and $x_{3-i}$ is a $3$-face.
Split $p$ into two vertices $u$ and $w$ along a simple curve from $f_1$ to $f_2$,
and add a vertex $v$ of degree four to the resulting 4-face bounded by the cycle $uy_1wy_2$.  
\item Amoeba expansion: Choose a non-proper amoeba in $F$ and replace its eye by any weak $4$-candidate with at least one internal vertex.
\end{itemize}

With this definition, Theorem~\ref{thm-redu-weak-cand} can be reformulated as follows.

\begin{corollary}\label{cor-generate-weak-onestep}
If Conjecture~\ref{conj-main-strong} holds, then
for every non-trivial restrictive weak $4$-candidate $G$, there exists a restrictive weak $4$-candidate $F$ such that
$G$ is an expansion of $F$.  Moreover, $F$ has the same kind as $G$.
\end{corollary}
\begin{proof}
Let us discuss each of the conclusions of Theorem~\ref{thm-redu-weak-cand} separately.

If $G$ contains a contractible bidiamond $\{u,v\}$, then let $F$ be obtained from $G$ by a contraction of this bidiamond,
so that $G$ is obtained from $F$ by a decontraction of a bidiamond.
Since the bidiamond is contractible, $F$ is a simple graph and every triangle in $F$ bounds a face, and thus $F$ is a weak $4$-candidate.
Moreover, every $4$-coloring $\varphi$ of $F$ extends to a $4$-coloring of $G$: 
Let $x$ and $y$ be the common neighbors of $u$ and $v$, let $u'$ and $v'$ be the neighbors of $u$ and $v$, respectively, not in $\{u,v,x,y\}$,
and let $p$ be the vertex of $F$ obtained by identifying $x$ with $y$.  Without loss of generality, we can assume that $\varphi(p)=1$, $\varphi(u')=2$,
and $\varphi(v')\in\{2,3\}$.  We give $x$ and $y$ the color $1$, $u$ the color $3$, and $v$ the color $4$.
Therefore, since $G$ is restrictive, $F$ is restrictive as well, and has the same kind.

If $G$ contains a $1$-contractible path $P$, then let $F$ be obtained from $G/P$ by deleting all (at most two) triangle-isolated vertices,
so that $F$ is a weak $4$-candidate.  Observe that $G$ is obtained from $F$ by a decontraction of a $1$-contractible path.
Moreover, by Observation~\ref{obs-reducopa}, the weak $4$-candidate $F$ is restrictive and has the same kind as $G$.

Finally, suppose that $H$ is a proper amoeba in $G$, with vertices labeled as in the definition of an amoeba.  We can choose $H$ so that the eye of $H$ is
as large as possible; thus, the vertices $y$ and $v_3$ do not have any common neighbor not belonging to the eye.  Let $F$ be obtained from $G$ by deleting
the internal vertices of the eye and adding the edge $yv_3$.  Then $F$ is a weak $4$-candidate and $G$ is obtained from $F$ by amoeba expansion.
Moreover, by Lemma~\ref{lemma-amoeba}, the weak $4$-candidate $F$ is restrictive (actually, even the canvas $F-yv_3$ is restrictive) and has the same kind as $G$.
\end{proof}

By induction, this has the following consequence.

\begin{corollary}\label{cor-generate-weak}
If Conjecture~\ref{conj-main-strong} holds, then
for every restrictive weak $4$-candidate $G$, there exists a sequence $G_0$, $G_1$, \ldots, $G_m=G$ of restrictive weak $4$-candidates of the same kind as $G$
such that $G_0$ is trivial and for every $i\in\{1,\ldots,m\}$, the weak $4$-candidate $G_i$ is an expansion of $G_{i-1}$.
\end{corollary}

Note that the diamond is the only trivial weak $4$-candidate of kind \texttt{d} and $W_4$ is the only trivial weak $4$-candidate of kind \texttt{r}.
Since the graphs in $\CC_0$ are exactly the graphs obtained from weak $4$-candidates by adding a crossing to the outer face, Corollary~\ref{cor-generate-weak}
implies Theorem~\ref{thm-gen4}, with the operation of expansion naturally lifted to the graphs in $\CC_0$.

\section{How to prove Conjecture~\ref{conj-main-strong}?}\label{sec-proofcom}

It is easy to see that Conjecture~\ref{conj-main-strong} is a strengthening of the Four Color Theorem (indeed, it is well-known and straightforward to prove
that a minimal counterexample to the Four Color Theorem would have minimum degree five and no separating triangles).  Up to some minor variations, we know
essentially only one proof of the Four Color Theorem: Using Kempe chains, prove that many configurations in planar graphs are \emph{reducible}, i.e.,
cannot appear in a minimal counterexample, then use discharging to show that at least one of them appears in every planar graph.
It is natural to ask whether a similar approach can be used to prove Conjecture~\ref{conj-main-strong}.  This indeed seems promising, but there are several
issues to overcome (we state them in the setting of restrictive 4-candidates, which is equivalent by Corollary~\ref{cor-combined}).
\begin{itemize}
\item Some of the Kempe chains may contain precolored vertices. This limits how we can exchange colors on Kempe chains,
and results in fewer of the configurations being reducible.
\item Reducibility arguments often involve replacing a configuration by a smaller one. This could decrease degrees of vertices or create non-facial
triangles, and thus we would not obtain a smaller counterexample and reach a contradiction.
\item The reducible configurations typically cannot contain precolored vertices, and this affects the discharging procedure.
\end{itemize}
The last issue is the least severe, since there is a standard technique for dealing with it by allocating extra charge for the precolored vertices.
We discuss the first two issues in more detail in the rest of this section.
Let us start with a useful observation.
\begin{lemma}\label{lemma-holefour}
Let $G$ be a restrictive weak $4$-candidate, let $K$ be a 4-cycle in $G$, and let $H$ be the subcanvas of $G$ drawn in the closed disk bounded by $K$.
If $H$ is not restrictive, then the graph $G_1$ obtained from $G$ by deleting the internal vertices of $H$ is restrictive.
Moreover, if $u,v\in V(K)$ are adjacent internal vertices whose degree in $G_1$ is four, then the graph $G_1-\{u,v\}$ is also restrictive.
\end{lemma}
\begin{proof}
Let $Q$ be the cycle bounding the outer face of $G$, and let $\psi$ be a 4-coloring of $Q$ which does not extend to a 4-coloring of $G$.
Suppose for a contradiction that $\psi$ extends to a 4-coloring $\varphi_1$ of $G_1$, and let $\psi_1$ be the restriction of $\varphi_1$ to $K$.
Then $\psi_1$ does not extend to a 4-coloring of $H$, and since $H$ is not restrictive, it follows that $\psi_1$ is bichromatic.
Let $K=v_1v_2v_3v_4$; we can assume that $\psi_1(v_1)=\psi_1(v_3)=1$ and $\psi_2(v_2)=\psi_2(v_4)=2$.

For any distinct colors $c_1$ and $c_2$,
let $G_{c_1,c_2}$ be the subgraph of $G_1$ induced by the vertices to which $\varphi_1$ assigns colors $c_1$ and $c_2$.
For a vertex $v$ such that $\varphi_1(v)\in\{c_1,c_2\}$, let $G_{c_1,c_2}(v)$ be the component of $G_{c_1,c_2}$ containing $v$.
If $G_{1,3}(v_1)$ contains $v_3$, then observe that by planarity, there exists $i\in\{2,4\}$ such that $G_{2,4}(v_i)$
is disjoint from $Q$ and does not contain $v_{6-i}$.  However, then we can exchange the colors $2$ and $4$ on $G_{2,4}(v_i)$,
turning $\varphi_1$ into a 4-coloring $\varphi_3$ of $G_1$ extending $\psi$ and whose restriction to $K$ is $v_1$-diagonal.  Since $H$ is not restrictive,
$\varphi_2$ further extends to a 4-coloring of $G$, which is a contradiction.  If $G_{1,3}(v_1)$ does not contain $v_3$ and does not intersect $Q$,
then we can similarly exchange colors $1$ and $3$ on $G_{1,3}(v_1)$ to turn $\varphi_1$ into a 4-coloring which is $v_2$-diagonal and extends $\psi$, then further extend it
to a 4-coloring of $G$, again obtaining a contradiction.

We conclude that $G_{1,3}(v_1)$ does not contain $v_3$ and intersects $Q$.  By symmetry, for every $i\in\{1,\ldots,4\}$ and $c\in\{3,4\}$,
the subgraph $G_{\psi_1(v_i),c}(v_i)$ intersects $K$ only in $v_i$ and intersects $Q$.  Note that this implies that
for every $x\in V(G_{1,3}(v_1)\cap Q)$ and $y\in V(G_{1,3}(v_3)\cap Q)$, the vertices $x$ and $y$ are distinct and non-adjacent.  Therefore, letting $Q=u_1u_2u_3u_4$, we can
assume that $\psi(u_1),\psi(u_3)\in \{1,3\}$ and $\psi(u_2),\psi(u_4)\in \{2,4\}$.

If $u_1\in V(G_{1,4}(v_1)\cap Q)\cup V(G_{1,4}(v_3)\cap Q)$, we similarly conclude that $\psi(u_1),\psi(u_3)\in \{1,4\}$ and $\psi(u_2),\psi(u_4)\in \{2,3\}$,
and thus $\psi(u_1)=\psi(u_3)=1$ and $\psi(u_2)=\psi(u_4)=2$.  Otherwise, we have $V(G_{1,4}(v_1)\cap Q)\cup V(G_{1,4}(v_3)\cap Q)=\{u_2,u_4\}$,
$\psi(u_2),\psi(u_4)\in\{1,4\}$ and $\psi(u_1),\psi(u_3)\in \{2,3\}$, and thus $\psi(u_1)=\psi(u_3)=3$ and $\psi(u_2)=\psi(u_4)=4$.
We conclude that $\psi$ is bichromatic, which is a contradiction.  Therefore, $G_1$ is restrictive.

Suppose now that say $v_1$ and $v_2$ are internal vertices whose degree in $G_1$ is four, and suppose for a contradiction that $\psi$ extends to a 4-coloring $\varphi'_1$
of $G_2=G_1-\{v_1,v_2\}$.  Since $G$ is a weak $4$-candidate, $G_1$ contains triangles $v_1v_4w_1$, $v_2v_3w_2$, $v_1v_2z$, $v_1w_1z$, and $v_2w_2z$ bounding faces,
and $K'=w_1zw_2v_3v_4$ is a cycle bounding a face of $G_2$.  Since $G_1$ is restrictive, $\varphi'_1$ cannot be extended to $v_1$ and $v_2$, and thus we can
assume that $\varphi'_1(z)=1$, $\varphi'_1(w_1)=\varphi'_1(v_3)=2$, and $\varphi'_1(w_2)=\varphi'_1(v_4)=3$.

For any distinct colors $c_1$ and $c_2$,
let $G'_{c_1,c_2}$ be the subgraph of $G_2$ induced by the vertices to which $\varphi'_1$ assigns colors $c_1$ and $c_2$.
For a vertex $v$ such that $\varphi'_1(v)\in\{c_1,c_2\}$, let $G'_{c_1,c_2}(v)$ be the component of $G'_{c_1,c_2}$ containing~$v$.
If $G'_{1,3}(v_4)$ contains $z$, then observe that by planarity, there exists $u\in\{v_3,w_1\}$ such that $G'_{2,4}(u)$
is disjoint from $Q$ and does not contain the vertex of $\{v_3,w_1\}\setminus \{u\}$.  However, then we can exchange the colors $2$ and $4$ on $G'_{2,4}(u)$,
turning $\varphi'_1$ into a 4-coloring of $G_2$ which extends $\psi$ and which can be extended to a 4-coloring of $G_1$, which is a contradiction.
Therefore, $z\not\in V(G'_{1,3}(v_4))$ and $v_4\not\in V(G'_{1,3}(z))$.  Let us remark that $w_2$ always belongs to $V(G'_{1,3}(z))$.
If $G'_{1,3}(v_4)$ or $G'_{1,3}(z)$ is disjoint from $Q$, then we can similarly
exchange the colors $1$ and $3$ on $G'_{1,3}(v_4)$ or $G'_{1,3}(z)$ and turn $\varphi'_1$ into a 4-coloring of $G_2$ which extends $\psi$ and which can be extended to a 4-coloring of $G_1$,
again obtaining a contradiction.

We conclude that $G'_{1,3}(v_4)$ and $G'_{1,3}(z)$ are disjoint and intersect $Q$.  By planarity, $G'_{2,4}(v_3)$ and $G'_{2,4}(w_1)$ are disjoint,
and a similar argument shows that they intersect $Q$.  Symmetrically, $G'_{1,2}(v_3)$ and $G'_{1,2}(z)$ are disjoint and intersect $Q$,
and $G'_{3,4}(v_4)$ and $G'_{3,4}(w_2)$ are disjoint and intersect $Q$.  As in the previous case, this is only possible if $\psi$ is bichromatic,
which is a contradiction.
\end{proof}

In particular, using this lemma, we can exclude separating 4-cycles from a minimal counterexample to Conjecture~\ref{conj-main-strong}.
\begin{corollary}\label{cor-no4}
Let $G$ be a restrictive 4-candidate and let $K$ be a $4$-cycle in $G$ that does not bound the outer face.
If no 4-candidate with less that $|V(G)|$ vertices is restrictive, then the cycle $K$ is hollow.
\end{corollary}
\begin{proof}
Let $Q$ be the 4-cycle bounding the outer face of $G$.
Suppose for a contradiction that the subcanvas $H$ of $G$ drawn in the closed disk bounded by $K$ has an internal vertex.
Let $G_1$ be the subgraph of $G$ obtained by deleting the internal vertices of $H$, and let $f$ be the 4-face of $G_1$ bounded by $K$.
We can assume that the 4-cycle $K$ is chosen so that $H$ is maximal, and thus no $4$-cycle in $G_1$ other than $Q$ and $K$ bounds
an open disk containing $f$.  If $K$ contains adjacent internal vertices $u$ and $v$ whose degree in $G_1$ is four, then
let $G_2=G_1-\{u,v\}$, otherwise let $G_2=G_1$.  By Lemma~\ref{lemma-holefour}, the graph $G_2$ is restrictive.

Next, iteratively repeat the following procedure:
Let $f'$ be the face of $G_2$ containing $f$.  If $G_2$ contains an induced path $P$ of length at most two with ends incident with
$f'$ and otherwise disjoint from the boundary of $f'$, then let $\Delta$ be the minimal closed disk containing $f'$ and $P$,
delete from $G_2$ the vertices and edges drawn in the interior of $\Delta$, and repeat.

Let $G_3$ be the resulting graph.  Using Observation~\ref{obs-pre3}, it is easy to see that every 4-coloring of $G_3$ extends to a 4-coloring of $G_2$,
and thus $G_3$ is restrictive.  Furthermore, observe that
\begin{itemize}
\item all faces of $G_3$ except for $f'$ and the outer one are triangles,
\item all internal vertices of $G_3$ not incident with $f'$ have degree at least five,
\item the face $f'$ is bounded by an induced cycle $K'$ and no vertex of $V(G_3)\setminus V(K')$ has non-adjacent neighbors in $K'$,
and in particular all internal vertices of $K'$ have degree at least four, and
\item either $G_3=G_1$, or $K'$ has length at least five.
\end{itemize}
If $G_3\neq G_1$, then let $G'$ be the graph obtained from $G_3$ by adding a vertex adjacent to all vertices of $K'$ and drawing it inside $f'$.
Then $G'$ is a restrictive 4-candidate with fewer vertices than $G$, which is a contradiction.

Hence, suppose that $G_3=G_1$. The construction of $G_2$ implies that no two internal vertices of $K$ of degree four are adjacent.
Hence, there exist non-adjacent vertices $x,y\in V(K)$
such that neither of the two vertices of $V(K)\setminus\{x,y\}$ is an internal vertex of degree four.  In this case, we let $G'=G_3+xy$.
Again, observe that $G'$ is a restrictive 4-candidate with fewer vertices than $G$, which is a contradiction.
\end{proof}

This allows us to exclude reducible configurations which do not need to be replaced by smaller ones.  More precisely, the following
claim holds.
\begin{corollary}\label{cor-dredu}
Let $G$ be a restrictive 4-candidate and let $K$ be a cycle in $G$ such that the open disk $\Lambda$ bounded by $K$ contains at least two vertices.
Let $G_1$ be the graph obtained from $G$ by deleting the vertices and edges drawn in $\Lambda$.  If no 4-candidate with less that $|V(G)|$ vertices is restrictive, then
$G_1$ is not restrictive.
\end{corollary}
\begin{proof}
Suppose for a contradiction that $G_1$ is restrictive.  Iteratively repeat the following procedure:
Let $f$ be the face of $G_1$ containing $\Lambda$.  If $G_1$ contains an induced path $P$ of length at most two with ends incident with
$f$ and otherwise disjoint from the boundary of $f$, then let $\Delta$ be the minimal closed disk containing $f$ and $P$,
delete from $G_1$ the vertices and edges drawn in the interior of $\Delta$, and repeat.

Let $G_2$ be the resulting graph.  By Observation~\ref{obs-pre3}, every 4-coloring of $G_2$ extends to a 4-coloring of $G_1$,
and thus $G_2$ is restrictive.
Moreover, all faces of $G_2$ except for $f$ and the outer one are triangles, all internal vertices of $G_2$ not incident with $f$ have degree at least five,
the face $f$ is bounded by an induced cycle $K'$ and no vertex of $V(G_2)\setminus V(K')$ has non-adjacent neighbors in $K'$,
and in particular all internal vertices of $K'$ have degree at least four.
Moreover, Corollary~\ref{cor-no4} implies that $K'$ has length at least five.

Let $G'$ be obtained from $G_2$ by adding a vertex adjacent to all vertices of $K'$ and drawing it inside $f$.
Then $G'$ is a restrictive 4-candidate with fewer vertices than $G$, which is a contradiction.
\end{proof}
Let $\theta$ be a (diagonal or rainbow) 4-coloring of the cycle $Q$ bounding the outer face of $G$ which does not extend to a 4-coloring of $G$,
and let $H$ be the subcanvas of $G$ drawn in the closure of $\Lambda$.
Corollary~\ref{cor-dredu} implies that if we can show (using Kempe chains and the properties of $H$)
that every 4-coloring of $G_1$ extending $\theta$ can be turned into a 4-coloring of $G$ with the same restriction to $Q$, then we obtain a contradiction.
Since the procedure for obtaining such proofs (showing that $H$ is \emph{D-reducible}) is standard (see e.g.~\cite{rsst} for a detailed description) and requires only minor modifications
to fit our setting, we leave its presentation to Appendix.

As shown in~\cite{steinberger2010unavoidable}, D-reducibility arguments are sufficient to prove the Four Color Theorem.  However, many of the configurations
that are D-reducible in the setting of the Four Color Theorem are not D-reducible in our setting, since we get extra constraints on switching
Kempe chains containing the precolored vertices (in fact, small D-reducible configurations seem to be very rare---there is some hope that
for large configurations, the loss of a few Kempe chains because of the precolored vertices becomes less important).

Realistically,
to prove the conjecture, we will need to use \emph{C-reducible} configurations, where the configuration $H$ is replaced by a smaller
configuration $H'$ (thus, we only need to argue that every 4-coloring of the resulting graph $G'$ which extends $\theta$ can be turned into a 4-coloring of $G$ with the same restriction to $Q$);
note that this replacement can also involve identification of the vertices of the cycle $K$ bounding the configuration.
An issue that proofs of the Four Color Theorem using $C$-reducible configurations need to overcome is that such a replacement could result in
creation of loops, making the reduced graph impossible to color (this is dealt with by observing that the loops correspond to short separating cycles in the original
graph, which are excluded using arguments similar to the proof of Corollary~\ref{cor-no4}).
In our setting, the issue is compounded by the possibility that $G'$ is not a candidate, i.e., that the replacement of $H$ by $H'$ can create internal vertices of degree at most four
or non-facial triangles (and no reduction along the lines of Corollary~\ref{cor-dredu} may be possible).  One way to deal with the issue would be to
find an exact (or at least sufficiently detailed) characterization of restrictive weak 4-candidates based on the assumption that the conjecture holds
for graphs with less than $|V(G)|$ vertices, and then use it to argue that $G'$ cannot be restrictive even if it is not a candidate.

\bibliographystyle{plain}
\bibliography{../../data.bib}

\section*{Appendix}

In this Appendix, we give a precise description of the standard Kempe chain reducibility argument in our setting.

A \emph{color partition} is an equivalence $\pi$ on the set $\{1,\ldots,4\}$ which
has exactly two equivalence classes of size two.  Let $Q$ and $K$ be disjoint cycles drawn
in the plane so that $K$ is contained in the closed disk bounded by $Q$, and let $\psi$ be a 4-coloring of $Q\cup K$.
Let $\psi/\pi$ be the equivalence on $V(Q\cup K)$ such that $u\equiv_{\psi/\pi} v$ if and only if $\psi(u)\equiv_\pi \psi(v)$.
An edge $uv$ of $Q\cup K$ is \emph{$\psi/\pi$-monochromatic} if $u\equiv_{\psi/\pi} v$.
A \emph{potential $\pi$-Kempe chain extract} for $\psi$ is a triple $(\beta,\kappa,\sigma)$ of equivalences on $V(Q\cup K)$
such that
\begin{itemize}
\item[(i)] $\beta=\psi/\pi$,
\item[(ii)] $\kappa$ is a refinement of $\psi/\pi$ such that $u\equiv_\kappa v$ for every $\psi/\pi$-monochromatic edge $uv\in E(Q\cup K)$,
\item[(iii)] $\sigma$ is the refinement of $\kappa$ such that for every $u,v\in V(Q\cup K)$, we have $u\equiv_\sigma v$ if and only if $u\equiv_\kappa v$ and $\psi(u)=\psi(v)$, and
\item[(iv)] there exists a plane graph $F$ consisting of a perfect matching $M$ on the non-$\psi/\pi$-monochromatic edges of $E(Q\cup K)$ and possibly of a cycle $C$ vertex-disjoint from the matching,
such that
\begin{itemize}
\item the drawing of $F$ is contained in the annulus $\Theta$ between $K$ and $Q$ and intersects the boundary of $\Theta$ only in the points representing the vertices of $M$,
equal to the midpoints of the drawings of the corresponding edges of $Q\cup K$, and
\item any two vertices of $V(Q\cup K)$ are $\kappa$-equivalent if and only if they are drawn in the same arcwise-connected component of the space obtained from $\Theta$ by removing
the drawing of $F$.
\end{itemize}
\end{itemize}
We say that such a plane graph $F$ is a \emph{realizer} for the potential $\pi$-Kempe chain extract.   Let us remark that whether such a realizer exists
can of course be determined from the combinatorial properties of $\kappa$.  Informally, the triple $(\beta,\kappa,\sigma)$ describes a possibility how Kempe chains of a $4$-colored plane graph
with facial cycles $K$ and $Q$ could look like:
\begin{itemize}
\item We simultaneously consider Kempe chains for two disjoint pairs $A=\{c_1,c_2\}$ and $B=\{c_3,c_4\}$ of colors, and $\beta$ divides
$V(Q\cup K)$ into vertices whose color belongs to $A$ and vertices whose colors belong to $B$ (without explicitly saying which equivalence class
corresponds to $A$ and which to $B$---this does not matter, since the two possibilities are equivalent up to a permutation of colors).
\item The classes of the equivalence $\kappa$ correspond to the intersections of the Kempe chains with $V(Q\cup K)$; the condition (iv) expresses the
planarity constraints on Kempe chains, as we discuss below.
\item The equivalence $\sigma$ divides each class of $\kappa$ into (at most) two parts based on the colors of vertices
(without explicitly specifying the colors corresponding to the equivalence classes, since the colors can be exchanged on the Kempe chain).
\end{itemize}

An \emph{environment} is a triple $(G,Q,K)$, where $G$ is a plane graph such that
the cycle $Q\subseteq G$ bounds the outer face of $G$, the cycle $K\subseteq G$ bounds an internal face of $G$, and all other faces of $G$ have length three.
Let $\varphi$ be a 4-coloring of $G$.  We are going to consider the Kempe chains of this 4-coloring for two disjoint pairs $\{c_1,c_2\}$ and $\{c_3,c_4\}$ of colors.
Let $\pi$ be the color partition such that $c_1\equiv_\pi c_2$ and $c_3\equiv_\pi c_4$, let $\psi$ be the restriction of $\varphi$ to $Q\cup K$, and
let $\beta=\psi/\pi$.  Let $X\subseteq E(G)$ consist of the edges $uv$ such that $\varphi(u)\not\equiv_\pi\varphi(v)$,
so that $G-X$ is the disjoint union of the $\{c_1,c_2\}$ and $\{c_3,c_4\}$ Kempe chains of $\varphi$, and let $\kappa$ be the equivalence on $V(Q\cup K)$
where two vertices are equivalent if and only if they belong to the same component of $G-X$.  Let $\sigma$ be the equivalence on $V(Q\cup K)$ such that
two vertices are equivalent if and only if they are $\kappa$-equivalent and have the same color.

We claim that $(\beta,\kappa,\sigma)$ is a potential $\pi$-Kempe chain extract for $\psi$.  Indeed, the conditions (i), (ii), and (iii) are clearly satisfied,
and thus it suffices to find a realizer showing that (iv) holds.
Let $F_1$ be the spanning subgraph of the dual of $G$ containing exactly the duals of the edges of $X$, and let $k$ and $q$ be the vertices of $F_1$ corresponding to the faces of $G$
bounded by $K$ and $Q$, respectively.  Observe that all vertices of $F_1$ have even degree; in particular, all vertices other than $k$ and $q$ have degree $0$ or $2$.
Hence, the drawing of $F_1$ can be decomposed into simple curves from $k$ to $q$ and simple closed curves passing through at most one of $k$ and $q$,
disjoint everywhere except for $k$ and $q$.  Moreover, note that that each $\{c_1,c_2\}$ and $\{c_3,c_4\}$ Kempe chain consists exactly of the vertices
of $G$ drawn in a face of $F_1$.  Thus, $F_1$ can be turned into a realizer of $(\beta,\kappa,\sigma)$ by
\begin{itemize}
\item deleting all components not containing $k$ and $q$, except for one cycle separating $k$ from $q$ if any is present,
\item subdividing each edge $e^\star$ of $F_1$ dual to an edge $e$ of $Q\cup K$ by a vertex drawn at the intersection of these edges, and
\item deleting the vertices $k$ and $q$ and suppressing the vertices of degree two.
\end{itemize}
We say that $(\beta,\kappa,\sigma)$ is the \emph{$\pi$-Kempe chain extract} of $\varphi$.

For a potential $\pi$-Kempe chain extract $(\beta,\kappa,\sigma)$ for a 4-coloring $\psi$ of $Q\cup K$,
let $n_{\psi,\pi,\beta,\kappa,\sigma}(G)$ denote the number of 4-colorings of $G$ which extend $\psi$ and whose $\pi$-Kempe chain extract is equal to $(\beta,\kappa,\sigma)$.
Let us note the following key observation.

\begin{observation}\label{obs-kempe-equal}
Let $(G,Q,K)$ be an environment, and for $i\in \{1,2\}$, let $\pi_i$ be a color partition and $\psi_i$ a $4$-coloring of $Q\cup K$.
Suppose that $(\beta,\kappa,\sigma)$ is both a potential $\pi_1$-Kempe chain extract for $\psi_1$ and a potential $\pi_2$-Kempe chain extract for $\psi_2$.
Then
$$n_{\psi_1,\pi_1,\beta,\kappa,\sigma}(G)=n_{\psi_2,\pi_2,\beta,\kappa,\sigma}(G).$$
\end{observation}
\begin{proof}
By permuting the colors if needed, we can assume that $\pi_1$ and $\pi_2$ are both equal to the color partition $\pi$ such that $1\equiv_\pi 2$ and $3\equiv_\pi 4$.
Since $\psi_1/\pi=\psi_2/\pi=\beta$, by exchanging the colors $1$ and $2$ with colors $3$ and $4$ in $\psi_2$ if needed,
we can furthermore assume that $\psi_1(v)\equiv_\pi \psi_2(v)$ for every $v\in V(Q\cup K)$.  Because $\psi_1$ and $\psi_2$ also share the part $\sigma$ of the potential $\pi$-Kempe chain extract,
this implies that there exists a set $S$ of equivalence classes of $\kappa$ such that $\psi_2$ is obtained from $\psi_1$ by exchanging the color $1$ with $2$ and the color $3$ with $4$ on the vertices
of the elements of $S$.

Let $\varphi$ be any $4$-coloring of $G$ extending $\psi_1$ whose $\pi$-Kempe chain extract is $(\beta,\kappa,\sigma)$.
For distinct colors $c_1$ and $c_2$, let $G_{c_1,c_2}$ be the subgraph of $G$ induced by the vertices of colors $c_1$ and $c_2$.
The definition of the $\pi$-Kempe chain extract implies that for each equivalence class $C$ of $\kappa$, there exists a component $G_C$
of $G_{1,2}\cup G_{3,4}$ intersecting $V(Q\cup K)$ exactly in $C$.  By exchanging the color $1$ with $2$ and the color $3$ with $4$ on $\bigcup_{C\in S} G_C$,
we transform $\varphi$ into a $4$-coloring of $G$ extending $\psi_2$ whose $\pi$-Kempe chain extract is $(\beta,\kappa,\sigma)$.  It is easy to see
that this correspondence establishes a bijection proving the equality
$$n_{\psi_1,\pi_1,\beta,\kappa,\sigma}(G)=n_{\psi_2,\pi_2,\beta,\kappa,\sigma}(G).$$
\end{proof}

Thus, let us define $n_{\beta,\kappa,\sigma}(G)$ as $n_{\psi,\pi,\beta,\kappa,\sigma}(G)$ for any 4-coloring $\psi$ and color partition $\pi$ such that $(\beta,\kappa,\sigma)$ is
a potential $\pi$-Kempe chain extract for $\psi$.  Let us also define $\Psi$ as the set of all 4-colorings of $K\cup Q$ and $\mathcal{E}$ as the set of all potential Kempe chain extracts,
and for every $\psi\in\Psi$, let $n_\psi(G)$ be the number of 4-colorings of $G$ that extend $\psi$.

Consider now an environment $(G,Q,K)$ and a plane graph $F$ with the outer face bounded by $K$.
Suppose that $\theta$ is a $4$-coloring of $Q$ that does not extend to 4-coloring of $G\cup F$.
Thus, $n_\psi(G)=0$ for every $4$-coloring $\psi$ of $Q\cup K$ that extends $\theta$.
Together with Observation~\ref{obs-kempe-equal} and the obvious fact that the number of 4-colorings subject to any condition is non-negative,
this gives a set of linear constraints satisfied by the vector $(n_i(G):i\in \Psi\cup \mathcal{E})$.

More precisely, let $\Pi$ be the set of all three color partitions, and for $\psi\in\Psi$ and $\pi\in\Pi$,
let $E_{\psi,\pi}\subseteq\mathcal{E}$ denote the set of all potential $\pi$-Kempe chain extracts for $\psi$.
For a 4-coloring $\theta$ of $Q$, let $\Psi_{F,\theta}$ be the set of 4-colorings $\psi\in \Psi$
such that the restriction of $\psi$ to $Q$ is $\theta$ and such that $\psi$ extends to a 4-coloring of $F$.
We define $\Phi_{F,\theta}$ as the cone (independent of the environment $G$) consisting of the points $x\in\mathbb{R}^{\Psi\cup \mathcal{E}}$ 
satisfying the following constraints:
\begin{align*}
x&\ge 0\\
x(\psi)&=\sum_{\epsilon \in E_{\psi,\pi}} x(\epsilon)&\text{ for every $\psi\in\Psi$ and $\pi\in \Pi$}\\
x(\psi)&=0&\text{for every $\psi\in\Psi_{F,\theta}$.}
\end{align*}
A coloring $\omega$ of $K$ is \emph{$(F,\theta)$-feasible} if there exists $x\in\Phi_{F,\theta}$ such that $x(\omega\cup\theta)\neq 0$;
that is, the set of constraints defining $\Phi_{F,\theta}$ does not exclude the possibility that $G$ has a 4-coloring extending $\theta$
whose restriction to $K$ equals $\omega$.  Otherwise, we say that $\omega$ is \emph{$(F,\theta)$-infeasible}.
In particular, every 4-coloring of $K$ which extends to a 4-coloring of $F$ is $(F,\theta)$-infeasible.

Let $F_1$ be a plane graph with the outer face bounded by a closed walk $W$ of the same length as $K$
and let $f$ be a function mapping the vertices of $K$ to vertices of $W$ in order.  Note that for a 4-coloring $\omega_1$ of $W$,
$f\circ\omega_1$ is a 4-coloring of $K$.  We say that $(F_1,f)$ is a \emph{$\theta$-reducent} for $F$ if
for every 4-coloring $\omega_1$ of $W$ which extends to a 4-coloring of $F_1$, the 4-coloring $f\circ\omega_1$ is $(F,\theta)$-infeasible.
Let $G_1$ be the graph obtained from $G$ by adding $F_1$ and identifying each vertex $v\in V(K)$ with the vertex $f(v)$;
we say that $G_1$ is obtained from $G\cup F$ by \emph{replacing} $F$ by $(F_1,f)$.
The basic property used to show reducibility of configurations is given in the following observation.  

\begin{observation}\label{obs-cone}
Let $G_0$ be a canvas with the outer face bounded by a cycle $Q$ and let $K$ be a cycle in $G$ disjoint from $Q$.
Let $F$ be the subgraph of $G_0$ drawn in the closed disk bounded by $K$ and let $G$ be the subgraph of $G_0$
drawn in the closed annulus between $K$ and $Q$, so that $(G,Q,K)$ is an environment.  Let $\theta$ be a 4-coloring of $Q$
and let $(F_1,f)$ be a $\theta$-reducent for $F$.  Let $G_1$ be the graph obtained from $G_0$ by replacing $F$ by $(F_1,f)$.
If $\theta$ does not extend to a 4-coloring of $G_0$, it also does not extend to a 4-coloring of $G_1$.
\end{observation}
\begin{proof}
Since $\theta$ does not extend to a 4-coloring of $G_0$, we have $n_\psi(G)=0$ for every $\psi\in\Psi_{F,\theta}$.
Moreover, for every $\psi\in\Psi$ and $\pi\in \Pi$, every 4-coloring of $G$ which extends $\psi$ has a uniquely determined $\pi$-Kempe chain extract,
and thus $$n_\psi(G)=\sum_{\epsilon \in E_{\psi,\pi}} n_{\psi,\pi,\epsilon}(G)=\sum_{\epsilon \in E_{\psi,\pi}} n_\epsilon(G).$$
Therefore, the point $x\in\mathbb{R}^{\Psi\cup \mathcal{E}}$ such that
$x(\psi)=n_\psi(G)$ for $\psi\in\Psi$ and $x(\epsilon)=n_{\epsilon}(G)$ for $\epsilon\in\mathcal{E}$ belongs to $\Phi_{F,\theta}$.

Suppose for a contradiction that $\varphi_1$ is a 4-coloring of $G_1$ extending $\theta$, and let $\omega_1$ be the restriction of $\varphi_1$ to $W$.
Clearly $\omega_1$ extends to a 4-coloring of $F_1$, and since $(F_1,f)$ is a $\theta$-reducent, the 4-coloring $f\circ\omega_1$ of $K$ is $(F,\theta)$-infeasible.

Let $\varphi$ be the 4-coloring of $G$ corresponding to $\varphi_1$; that is, $\varphi(v)=\varphi_1(v)$ for $v\in V(G)\setminus V(K)$ and
$\varphi(v)=\varphi_1(f(v))$ for $v\in V(K)$.  Let $\psi$ be the restriction of $\varphi$ to $Q\cup K$.
Since $\psi$ extends to a 4-coloring $\varphi$ of $G$, we have
$$0<n_\psi(G)=x(\psi).$$
Observe that $\psi=\theta\cup (f\circ\omega_1)$, and thus the 4-coloring $f\circ\omega_1$ is $(F,\theta)$-feasible.  This is a contradiction.
\end{proof}

In the situation of Observation~\ref{obs-cone}, if $|V(F)|>|V(K)|$ and $F_1=K$ (i.e., all 4-colorings of $K$ are $(F,\theta)$-infeasible), then we say that $F$ is \emph{D-reducible}.
In case that we know that $\theta$ extends to a 4-coloring of every non-trivial environment in $G_0$ (e.g., as in Corollary~\ref{cor-dredu}), this implies that
the configuration $F$ cannot appear in $G_0$.  If $F_1$ is any graph with $|V(F_1)|<|V(F)|$, then $F$ is \emph{C-reducible with reducent $F_1$}.
As we discuss at the end of Section~\ref{sec-proofcom}, there are extra challenges in being able to use C-reducible configurations, since it is harder to
ensure that induction hypothesis can be applied to $G_1$.

Finally, let us remark that in the presented formulation (called block-count reducibility~\cite{bcr1,bcr2}), determining whether a coloring is $(F,\theta)$-feasible involves solving a linear program;
this is generally too slow to be practical for large configurations.  Instead, one usually solves the boolean restriction of the program, with variables $x(\psi)$
and $x(\epsilon)$ being true when $n_\psi(G)>0$ and $n_{\epsilon}(G)>0$, respectively, and with addition replaced by logical disjunction.
In this boolean formulation, the program can be solved by straightforward greedy propagation, see e.g.~\cite{rsst} for more details.

\end{document}